\documentclass[12pt]{amsart}
\usepackage{a4}
\usepackage{booktabs}
\usepackage[T1]{fontenc}
\usepackage[all]{xy}
\usepackage{lipsum}
\usepackage{url}
\usepackage{tikz-network}
\usepackage{stackrel}
\usepackage{color}

\usepackage{amsmath,amsthm,amssymb}
\newcommand{\aspas}[1]{``{#1}''}
\usepackage{xcolor}
\usepackage{graphicx}
\usepackage{tikz}
\usetikzlibrary{
  hobby,
  intersections,
  spath3,
  decorations.markings,
  arrows.meta,
}

\newtheorem{theorem}{Theorem}[section]
\newtheorem{lemma}[theorem]{Lemma}
\newtheorem{example}[theorem]{Example}
\newtheorem{proposition}[theorem]{Proposition}
\theoremstyle{definition}
\newtheorem{definition}[theorem]{Definition}

\newtheorem{remark}[theorem]{Remark}
\newtheorem{corollary}[theorem]{Corollary}

\numberwithin{equation}{section}

\begin{document}


\renewcommand{\bf}{\bfseries}
\renewcommand{\sc}{\scshape}

\title[(Injective) hom-complexity between graphs]%
{(Injective) hom-complexity between graphs}

\author[C. A. Ipanaque Zapata]{Cesar A. Ipanaque Zapata}
\address[C. A. Ipanaque Zapata]{Departamento de Matem\'atica - IME-USP,
Caixa Postal 66281 - Ag. Cidade de S\~{a}o Paulo,
CEP: 05314-970 - S\~{a}o Paulo - SP - Brasil}
\email{cesarzapata@usp.br}

\author[J.A. Aguirre Enciso]{Josué A. Aguirre Enciso}%
\address[J.A Aguirre Enciso]{Departamento de Ciencias, Pontificia Universidad Cat\'{o}lica del Per\'{u},   Av. Universitaria 1801, San Miguel, 15088,  Per\'{u}\\ Facultad de Ciencias Matemáticas, Universidad Nacional Mayor de San Marcos,   Lima 15081}
\email{josue.aguirre@pucp.edu.pe}
\email{jaguirree@unmsm.edu.pe} %

\author[W.F. Cuba Ramos]{Wilman Francisco Cuba Ramos}%
\address[W.F. Cuba Ramos]{Facultad de Ciencias Matemáticas, Universidad Nacional Mayor de San Marcos, Lima 15081, Per\'{u}}
\email{wfcubar@gmail.com} %



\begin{abstract} We present the notion of hom-complexity, $\text{C}(G;H)$, for two graphs $G$ and $H$, along with basic results for this numerical invariant. This invariant $\text{C}(G;H)$ is a number that measures the \aspas{complexity} of the question: when is there a homomorphism $G\to H$? More precisely, $\text{C}(G;H)$ is the least positive integer $k$ such that there are $k$ different subgraphs $G_j$ of $G$ such that $G=G_1\cup\cdots\cup G_k$, and for each $G_j$, there is a homomorphism $G_j\to H$. Likewise, we introduce the notion of injective hom-complexity, $\text{IC}(G;H)$. The (injective) hom-complexity is a graph invariant. Additionally, these invariants can be used to show the nonexistence of homomorphisms. We explore the sub-additivity of (injective) hom-complexity and study products. 

We describe bounds for the hom-complexity in terms of chromatic number $\chi$ and clique number $\omega$. We provide the formula \[\text{C}(G;H)=\lceil\log_{\chi(H)}\chi(G)\rceil\]  whenever $\omega(H)=\chi(H)$. For example, we obtain $\text{C}(G;K_\ell)=\lceil\log_{\ell}\chi(G)\rceil$. Moreover, we discuss a connection between the (injective) hom-complexity and several well-known covering numbers. For instance, we provide a lower bound for the clique covering number in terms of the injective hom-complexity. Additionally, we show that the hom-complexity $\mathrm{C}(G;K_{\ell})$ coincides with the $\ell$-particity $\beta_\ell(G)$ of $G$, and the hom-complexity $\mathrm{C}(K_n;K_{2})$ coincides with the bipartite dimension $\mathrm{d}(K_n)$ of $K_n$. As a consequence, we recover the well-known formulas $\beta_\ell(G)=\lceil\log_{\ell}\chi(G)\rceil$ and $\mathrm{d}(K_n)=\lceil\log_{2}n\rceil$.  
\end{abstract}
\maketitle

\section{Introduction}\label{secintro}%
In this article, the term \aspas{graph} refers to a graph (either directed or undirected, depending on the context) that does not allow multiple edges (i.e., more than one edge connecting two vertices). As usual, $K_n$ denotes the complete graph on $n$ vertices. A simple graph $G$ admits a $k$-colouring if and only if there exists a vertex-surjective homomorphism $G\to K_k$, as explained in Section~\ref{sec:pre}. Likewise, $G$ is $\ell$-partite if and only if there exists a vertex-surjective homomorphism $G\to K_\ell$ (see Definition~\ref{defn:l-partite}(1)). For more details, see Section~\ref{sec:pre}. The symbol $\lceil m\rceil$ denotes the least integer greater than or equal to $m$, while $\lfloor m\rfloor$ denotes the greatest integer less than or equal to $m$. 

\medskip Let $G$ and $H$ be graphs. A \textit{homomorphism} of $G$ to $H$, written as $f:G\to H$, is a mapping $f:V(G)\to V(H)$ such that $f(u)f(v)\in E(H)$ whenever $uv\in G$. Thus, homomorphisms of undirected graphs preserve adjacency, while homomorphsims of directed graphs also preserve the directions of the arcs. The symbol $G\to H$ means that there is a homomorphism from $G$ to $H$, and in this case, we say that $G$ is \textit{$H$-colourable}; otherwise, we write $G\not\rightarrow H$. A \textit{$k$-colouring} of a simple graph $G$ is an assignment of $k$ colours to the vertices of $G$ such that adjacent vertices have different colours. Observe that the statement \aspas{$G$ admits a $k$-colouring} is not equivalent to \aspas{$G$ is $K_k$-colourable}. For instance, $K_2$ is $K_3$-colourable (by the inclusion homomorphism), but $K_2$ does not admit a $3$-colouring. 

\medskip We introduce the notion of hom-complexity between two graphs $G$ and $H$, denoted by $\text{C}(G;H)$ (Definition~\ref{defn:complexity}), along with its basic results. More precisely, $\text{C}(G;H)$ is defined as the least positive integer $k$ such that there are $k$ distinct subgraphs $G_j$ of $G$ with $G=G_1\cup\cdots\cup G_k$, and over each $G_j$, there exists a homomorphism $G_j\to H$. For instance, we have $\text{C}(G;H)=1$ if and only if there is a homomorphism $G\to H$. Likewise, we introduce the notion of injective hom-complexity, $\text{IC}(G;H)$. The invariant $\text{C}(G;H)$, in view of Proposition~\ref{prop:hom-cov-number}, can also be called the \textit{$H$-covering number} of $G$. In contrast, $\text{IC}(G;H)$ cannot be viewed as a covering number, as observed in Remark~\ref{rem:in-contrast-ic}. 

\medskip The considerable interest that $\text{C}(G;H)$ has received is motivated by a fundamental desire to present a well-defined methodology to address the \aspas{complexity} of the data migration problem \cite{hussein2021}, \cite{spivak2012}: This invariant represents the effective size of a data migration from $G$ to $H$ in a data migration process, where the graphs $G$ and $H$ represent data models (schemes). In other words, we consider the initial data model $\mathfrak{G}$ and the desired data model $\mathfrak{H}$, where each data model (schema) consists of tables and columns. Data migration is the process of moving data from $\mathfrak{G}$ to $\mathfrak{H}$, where each table is mapped to a table and each column to a column. Each schema is represented as a graph, with each table as a vertex and each column as an edge (cf. \cite{spivak2012}). Thus, we obtain our graphs $G$ and $H$, and it can be said a data migration from $\mathfrak{G}$ to $\mathfrak{H}$ is a homomorphism from $G$ to $H$. This leads to the intuitive interpretation of $\text{C}(G;H)$ as the effective size of data migration (i.e., finding the least number of data migrations while decomposing the initial data model). We note that a previous version of the hom-complexity $\text{C}(-;-)$ appeared in the Ph.D. thesis (in Spanish) \cite{wilman} of the third author. 
 
\medskip On the other hand, given two graphs $G$ and $H$, it is natural to pose the following question: When is there a homomorphism $G\to H$? This question represents a significant challenge in graph theory (see \cite{hell1990}). In fact, as shown in \cite[Theorem 1, p. 93]{hell1990}, this problem is NP-complete whenever $H$ is not bipartite. Hence, calculating $\text{C}(G;H)$ is NP-complete in general.

\medskip Among the central problems in hom-complexity theory are the following:
\begin{itemize}
    \item computing the hom-complexity $\text{C}(G;H)$,
    \item realizing $\text{C}(G;H)$, that is, finding an explicit optimal quasi-homomorphism $\mathcal{M}=\{f_i:G_i\to H\}_{i=1}^k$ from $G$ to $H$ (see Section~\ref{sec:one}). 
\end{itemize} Likewise, for injective hom-complexity. 

\medskip We present the following remark to avoid confusion between the name hom-complexity, hom-complex, and computational complexity.

\begin{remark}[Hom-complexity vs. Hom-complex vs. Computational Complexity]
  Given two simple graphs $G$ and $H$, the \textit{Hom-complex} $\mathrm{Hom}(G,H)-$whose elements are all homomorphisms from $G$ to $H-$was introduced by Lov\'{a}sz to provide lower bounds for the chromatic number  of graphs \cite{lovasz1978}, \cite[Definition 1.2, p. 286]{badson2006}. Observe that $\text{C}(G;H)\geq 2$ if and only if $\mathrm{Hom}(G,H)=\varnothing$. Thus, the Hom-complex is a topological space (a complex), while the hom-complexity is a numerical value, in the spirit of \textit{computational complexity}, which refers to the difficulty of solving a problem in terms of the number of computational operations required.   
\end{remark}    

\medskip We discuss a connection between (injective) hom-complexity and some well-known covering numbers \cite{schwartz2022}. For instance, we prove that the hom-complexity $\mathrm{C}(G;K_{\ell})$ coincides with the $\ell$-particity $\beta_\ell(G)$ of $G$\footnote{Although \aspas{$K_\ell$-colourable} and \aspas{$\ell$-partite} are not the same thing, as explained in  Definition~\ref{defn:l-partite}(1).}, and the hom-complexity $\mathrm{C}(K_n;K_{2})$ coincides with the bipartite dimension $\mathrm{d}(K_n)$ of $K_n$\footnote{Although \aspas{$K_2$-colourable} and \aspas{biclique} are not the same thing, as explained in  Definition~\ref{defn:l-partite}.} (Theorem~\ref{prop:biparticity-complexity-k2}). As a consequence, we recover the formulas $\beta_\ell(G)=\lceil\log_{\ell}\chi(G)\rceil$ obtained by Harary-Hsu-Miller in \cite{harary1977}\footnote{The original proof of this was given for $\ell=2$ (see \cite[Theorem, p. 131]{harary1977}), and the authors observed that it holds for any $\ell$ (see \cite[Observation A, p. 132]{harary1977}).}, and $\mathrm{d}(K_n)=\lceil\log_{2}n\rceil$ obtained by Fishburn-Hammer in \cite[Lemma 1, p. 130]{fishburn1996}. These formulas are special cases of our formula $\mathrm{C}(G;H)=\lceil\log_{\chi(H)}\chi(G)\rceil$, which holds whenever $\omega(H)=\chi(H)$ (Corollary~\ref{cor:omega-equal-chi}).

\medskip As the Remark~\ref{rem:bipa-dim-dif-complex} suggests, the invariants $\mathrm{d}(-)$ and $\mathrm{C}(-;K_{2})$ may not coincide in general. This is significant because the hom-complexity $\mathrm{C}(-;K_{2})$ is more flexible in terms of handling subgraph structures, whereas $\mathrm{d}(-)$ seems more strictly defined by the biclique structure of the graph. 

\medskip The main results of this work are:
\begin{itemize}
    \item Introduction of the concepts of hom-complexity $\text{C}(G;H)$ and injective hom-complexity $\text{IC}(G;H)$ for two graphs $G$ and $H$ (Definition~\ref{defn:complexity}).
     \item A triangular inequality (Theorem~\ref{thm:inequality-three-graphs}). 
    \item We establishes a close connection between the hom-complexities and homomorphisms (Theorem~\ref{prop:complexity-subgraphs}). In particular, this shows  that (injective) hom-complexity is a graph invariant. It also can be used to show the nonexistence of homomorphisms. Furthermore, it implies that hom-complexity can be restricted to cores. 
    \item A formula for $\mathrm{IC}(G;K_2)$ (Theorem~\ref{thm:injective-category-compute}).
    \item Sub-additivity (Theorem~\ref{thm:category-union}).
    \item Product inequality (Theorem~\ref{thm:product-inequality-complexity}).
    \item A lower bound (Theorem~\ref{thm:lower-bound}).
    \item An upper bound for hom-complexity (Theorem~\ref{thm:general-upper-bound}).
    \item The formula $\mathrm{C}(G;H)=\lceil\log_{\chi(H)}\chi(G)\rceil$ whenever $\omega(H)=\chi(H)$ (Corollary~\ref{cor:omega-equal-chi}). In particular, we obtain $\mathrm{C}(G;K_\ell)=\lceil\log_{\ell}\chi(G)\rceil$ (Example~\ref{exem:complexity-kj-k2}). 
    \item An explicit construction of an optimal quasi-homomorphism $\{\hat{f}_i:\hat{G}_i\to K_\ell\}_{i=1}^{k}$ from $G$ to $K_\ell$ such that each $\hat{G}_i$ is a $\ell$-partite subgraph of $G$ (Theorem~\ref{prop:covering-l-partite}).
    \item The hom-complexity $\mathrm{C}(G;K_{\ell})$ coincides with the $\ell$-particity of $G$, and the hom-complexity $\mathrm{C}(K_n;K_{2})$ coincides with the bipartite dimension of $K_n$ (Theorem~\ref{prop:biparticity-complexity-k2}). 
\end{itemize} 

The paper is organized as follows: We begin with a brief review of graphs and homomorphisms (Section~\ref{sec:pre}). In Section~\ref{sec:one}, we introduce the notions of hom-complexity $\text{C}(G;H)$ and injective hom-complexity $\text{IC}(G;H)$ for two graphs $G$ and $H$ (Definition~\ref{defn:complexity}). Theorem~\ref{prop:complexity-subgraphs} establishes a close connection between the hom-complexities and homomorphisms. Proposition~\ref{prop:hom-complexity-core} shows that hom-complexity can be restricted to cores. A formula for $\mathrm{IC}(G;K_2)$ is presented in Theorem~\ref{thm:injective-category-compute}. Furthermore, Theorem~\ref{thm:category-union} demonstrates the sub-additivity of (injective) hom-complexity. Theorem~\ref{thm:product-inequality-complexity} establishes a product inequality. A lower bound in terms of chromatic number is provided in  Theorem~\ref{thm:lower-bound}. A key result for obtaining an upper for hom-complexity is Proposition~\ref{prop:upper-bound}, leading to the  upper bound presented in Theorem~\ref{thm:general-upper-bound}. The formula $\mathrm{C}(G;H)=\lceil\log_{\chi(H)}\chi(G)\rceil$ whenever $\omega(H)=\chi(H)$ is presented in Corollary~\ref{cor:omega-equal-chi}. In Theorem~\ref{prop:covering-l-partite}, we provide an optimal quasi-homomorphism $\{\hat{f}_i:\hat{G}_i\to K_\ell\}_{i=1}^{k}$ from $G$ to $K_\ell$ such that each $\hat{G}_i$ is a $\ell$-partite subgraph of $G$. 

\medskip In Section~\ref{sec:cov-number}, we discuss a connection between (injective) hom-complexity and some well-known covering numbers. Proposition~\ref{prop:hom-cov-number} shows that the hom-complexity coincides with a covering number. In contrast, the injective hom-complexity does not coincide with a covering number, as shown in Remark~\ref{rem:in-contrast-ic}. We recall the notions of clique covering number, $\ell$-particity, and $\ell$-partite dimension in Definition~\ref{defn:cc-l-particity-l-partite-dim}. Proposition~\ref{prop:complexity-lower-bound-cn}(1) presents a lower bound for the clique covering number $cc(G)$ of $G$ in terms of the injective hom-complexity. Example~\ref{exam:cc-w2} shows that this bound improves, for some examples, the lower bound $\log_2(|V(G)|+1)$ presented in \cite{gyarfas1990}. In addition, Proposition~\ref{prop:properties-cov-number} provides a new lower bound for $cc(G)$. Theorem~\ref{prop:biparticity-complexity-k2} states that the hom-complexity $\mathrm{C}(G;K_{\ell})$ coincides with the $\ell$-particity $\beta_\ell(G)$ of $G$, and the hom-complexity $\mathrm{C}(K_n;K_{2})$ coincides with the bipartite dimension $\mathrm{d}(K_n)$ of $K_n$. In Proposition~\ref{cor:biparticity-formula}, we recover the formulas for $\beta_\ell(G)$ and $\mathrm{d}(K_n)$. We close this section with Remark~\ref{rem:complexity-sectionalnumber}, which presents a direct connection between the (injective) hom-complexity and the notion of sectional number.  

\medskip In Section~\ref{sec:applica}, we design optimal quasi-homomorphisms for a concrete example of Data Migration. We close this section with Remark~\ref{rem:future-work}, which presents directions for future work based on a new notion, extension, or application of hom-complexity. 

\section{Graphs and homomorphisms revisited} \label{sec:pre}
In this section, we recall some definitions and we fix the notations. We follow the standard notation for graphs as used in \cite{hell2004}. A \textit{digraph} or \textit{directed graph} $G$ is a set $V=V(G)$ of \textit{vertices}, together with a binary relation $E=E(G)$ on $V$, i.e., $E\subseteq V\times V$. The elements $(u,v)$ of $E$ are called the \textit{arcs} of $G$. An \textit{undirected graph} $G$ is a set $V=V(G)$ of vertices together with a set $E=E(G)$ of edges, each of which is a two-element set of vertices. If we allow \textit{loops}, i.e., edges that consist of a single vertex, we have a graph \textit{with loops allowed}. In this article, the term \aspas{graph} refers to either  a directed or undirected graph, depending on the context. 

\medskip Two graphs $G$ and $H$ are equal if $V(G)=V(H)$ and $E(G)=E(H)$. 

\medskip We shall use the usual simplified notation for arcs and edges, where $uv$ represents the arc $(u,v)$, or the edge $\{u,v\}$, depending on the context. A loop at $u$ is written as $uu$. If $uv\in E(G)$, we say that $u$ and $v$ are \textit{adjacent}. If $G$ is an undirected graph, we have $uv=vu$. If $uv$ is an arc in a digraph, we say that $u$ and $v$ are adjacent \textit{in the direction from $u$ to $v$}, or that $u$ is an \textit{in-neighbour} of $v$, and $v$ is an \textit{out-neighbour} of $u$. In any case, $u$ and $v$ are adjacent in a digraph as long as at least one of $uv$ or $vu$ is an arc; in that case, we also say that $u$ and $v$ are \textit{neighbours}. The number of neighbours of $v$ (other than $v$) is called the \textit{degree} of $v$; the number of in-neighbours of $v$, and out-neighbours of $v$ is called the \textit{in-degree} and \textit{out-degree} of $v$, respectively. Furthermore, $\text{indeg}(v)$, $\text{outdeg}(v)$, and $\text{deg}(v)$ denote the in-degree, out-degree, and degree of vertex $v$, respectively. A vertex $v$ is called \textit{isolated} if $\text{deg}(v)=0$.  

\medskip We say that a graph $H$ is a \textit{subgraph} of $G$ if $V(H)\subseteq V(G)$ and $E(H)\subseteq E(G)$. A subgraph $H$ of $G$ is called a \textit{spanning subraph} if $V(H)=V(G)$. Additionally, $H$ is an \textit{induced subgraph} of $G$ if it is a subgraph of $G$ and contains all the arcs (edges) of $G$ among the vertices in $H$. We say that a graph $G$ is \textit{complete} if $E(G)=\{\{u,v\}:~u,v\in V(G) \text{ with } u\neq v\}$. A \textit{clique} in a graph $G$ is a complete subgraph of $G$. 

\medskip Let $G$ and $H$ be graphs. A \textit{homomorphism} of $G$ to $H$, written as $f:G\to H$, is a mapping $f:V(G)\to V(H)$ such that $f(u)f(v)\in E(H)$ whenever $uv\in G$. Hence, homomorphisms of undirected graphs preserve adjacency, while homomorphsims of digraphs also preserve the directions of the arcs. The symbol $G\to H$ indicates that there exists a homomorphism from $G$ to $H$, and in this case, we say that $G$ is \textit{$H$-colourable}; otherwise, we write $G\not\rightarrow H$ (cf. \cite[p. 3]{hell2004}). Note that if we remove or add isolated vertices from $G$, its $H$-colorability does not change. Given a subgraph $H$ of $G$, the inclusion map $V(H)\hookrightarrow V(G)$ is a homomorphism and is called the \textit{inclusion homomorphism} $G\hookrightarrow H$. 

\medskip A homomorphism $f:G\to H$ is a mapping from $V(G)$ to $V(H)$, but since it preserves adjacency, it also naturally defines a mapping $f^{\#}:E(G)\to E(H)$ by setting $f^{\#}(uv)=f(u)f(v)$ for all $uv\in E(G)$. We call a homomorphism $f:G\to H$ \textit{vertex-injective},   \textit{vertex-surjective}, or \textit{vertex-bijective} if the mapping $f:V(G)\to V(H)$ is injective, surjective, or bijective,  respectively; and \textit{edge-injective}, \textit{edge-surjective}, or \textit{edge-bijective} if the mapping $f^{\#}:E(G)\to E(H)$  is injective, surjective, or bijective, respectively. A homomorphism $f$ is an \textit{injective homomorphism}, a \textit{surjective homomorphism}, or a \textit{bijective homomorphism} if it is both vertex- and edge-injective, surjective, or bijective, respectively. Note that if $f:G\to H$ is a bijective homomorphism, the inverse map $f^{-1}:V(H)\to V(G)$ is a homomorphism from $H$ to $G$, and in this case, we say that $f:G\to H$ is a \textit{graph isomorphism}, and that $G$ and $H$ are \textit{isomorphic}.   

\medskip Note that a homomorphism that is vertex-injective is also edge-injective (but not conversely), and as long as $H$ has no isolated vertices, a homomorphism that is edge-surjective is also vertex-surjective (but not conversely). In other words, injective homomorphisms are the same as vertex-injective homomorphisms, while surjective homomorphisms are, in the absence of isolated vertices, the same as edge-surjective homomorphisms. 

\medskip The following statement is straightforward to verify.

\begin{lemma}\label{lem:hom-degree}
If $f:G\to H$ is an injective homomorphism and $v\in V(G)$. Then: 
\begin{enumerate}
    \item[(1)] $\mathrm{deg}(f(v))\geq \mathrm{deg}(v)$.
     \item[(2)] $\mathrm{indeg}(f(v))\geq \mathrm{indeg}(v)$.
      \item[(3)] $\mathrm{outdeg}(f(v))\geq \mathrm{outdeg}(v)$.
\end{enumerate}
\end{lemma}

Now, we recall the definition of the union of graphs.

\begin{definition}[Union of Graphs]\label{defn:union-graphs}\cite[Definition 1.2.20, p. 25]{west2001}
 Let $G_1, G_2,\ldots, G_k$ be graphs. The \textit{union} $G_1\cup\cdots\cup G_k$ is defined by $V(G_1\cup\cdots\cup G_k)=V(G_1)\cup\cdots\cup V(G_k)$, and $E(G_1\cup\cdots\cup G_k)=E(G_1)\cup\cdots\cup E(G_k)$.   
\end{definition}
 
Let $G$ be a graph and $A,B$ be subgraphs of $G$ such that $V(A)\cap V(B)=\varnothing$ (and thus $E(A)\cap E(B)=\varnothing$). In this case, the union $A\cup B$ is called the \textit{disjoint union} and is denoted by \begin{align}\label{eq:disjoint-union}
    A\sqcup B.
\end{align} Furthermore, given homomorphisms $f:A\to H$ and $g:B\to H$, the map $f\sqcup g:V(A)\cup V(B)\to V(H)$, defined by \[(f\sqcup g)(v)=\begin{cases}
    f(v),&\hbox{ if $v\in V(A)$,}\\
    g(v),&\hbox{ if $v\in V(B)$}, 
\end{cases}\] is a homomorphism of $A\sqcup B$ to $H$ (using the fact that $E(A)\cap E(B)=\varnothing$).  

\medskip A \textit{simple graph} is an undirected graph without loops (and, of course, without multiple edges). 
 
\medskip A \textit{$k$-colouring} of a simple graph $G$ is an assigment of $k$ colours to the vertices of $G$, such that adjacent vertices have different colours \cite[p. 6]{hell2004}. Denote by $K_k$ the complete graph on vertices $1,2,\ldots,k$, and suppose these integers are used as the \aspas{colours} in $k$-colourings. Then, a $k$-colouring of $G$ can be viewed as a surjective mapping $f:V(G)\to \{1,2,\ldots,k\}$; the requirement that adjacent vertices have distinct colours means that $f(u)\neq f(v)$ whenever $uv\in E(G)$. It remains to observe that the condition $f(u)\neq f(v)$ is equivalent to the condition $f(u)f(v)\in E(K_k)$, and we may conclude that $f:G\to K_k$ is a vertex-surjective homomorphism. On the other hand, if there exists a homomorphism $f:G\to K_k$ (i.e., $G$ is $K_k$-colourable), then $G$ admits a $n$-colouring with $n\leq k$. Note that $n=k$ whenever $f:G\to K_k$ is a vertex-surjective homomorphism (cf. \cite[Proposition 1.7, p. 7]{hell2004}. For example, the inclusion map $\{1,2\}\hookrightarrow\{1,2,3\}$ is a homomorphism from $K_2$ to $K_3$, i.e., $K_2$ is $K_3$-colourable, but $K_2$ does not admit a $3$-colouring.  

\medskip The \textit{chromatic number} of $G$, denoted by $\chi(G)$, is defined as the smallest $k$ such that $G$ admits a $k$-colouring. Note that if we remove or add isolated vertices from $G$, its chromatic number does not change. Furthermore, if $G\to H$, then $\chi(G)\leq\chi(H)$ \cite[Corollary 1.8, p. 7]{hell2004}. 

\medskip  We have the following statement, which is fundamental in Theorem~\ref{thm:lower-bound}.

\begin{proposition}\label{prop:chromatic-union}(cf. \cite[Exercise 11, p. 36]{hell2004}) 
  Let $G$ be a simple graph, and  let $G_1,\ldots,G_m$ be subgraphs of $G$ such that $G=G_1\cup\cdots\cup G_m$. Then, we have \[\chi(G)\leq \prod_{j=1}^m\chi(G_j).\]  
\end{proposition}
\begin{proof}
Since the chromatic number does not change when we add isolated vertices, we can assume that $V(G_i)=V(G)$ for each $i$ (i.e., each $G_i$ is a spanning subraph of $G$). Suppose that $\ell_i=\chi(G_i)$, and consider a homomorphism $f_i:G_i\to K_{\ell_i}$ for each $i$. Define the map $f:V(G)\to \{1,\ldots,\ell_1\}\times\cdots\times \{1,\ldots,\ell_m\}$ by \[f(v)=\left(f_1(v),\ldots,f_m(v)\right) \quad \text{ for all $v\in V(G)$.}\]  Since $f(u)\neq f(v)$ (indeed, $f_i(u)\neq f_i(v)$ for each $i$) whenever $uv\in E(G)$, $f$ is a homomorphism from $G$ to $K_{\ell_1\cdots\ell_m}$, where we consider $V\left(K_{\ell_1\cdots\ell_m}\right)=\{1,\ldots,\ell_1\}\times\cdots\times \{1,\ldots,\ell_m\}$. Hence, $\chi(G)\leq \ell_1\cdots\ell_m=\prod_{j=1}^m\chi(G_j)$.    
\end{proof}

The following notion will be fundamental in Proposition~\ref{prop:upper-bound}.

\begin{definition}[$\ell$-partite graph and $\ell$-clique]\label{defn:l-partite}
Let $\ell\geq 2$ be an integer and $G$ be a simple graph.
\begin{enumerate}
    \item[(1)]  $G$ is called \textit{$\ell$-partite} if its vertex set can be partitioned into $\ell$ disjoint nonempty subsets $V_1,\ldots,V_\ell$ such that each edge connects a vertex of $V_i$ with a vertex of $V_j$ for $i\neq j$. A $2$-partite graph is also called a \textit{bipartite graph}. Note that $G$ is $\ell$-partite if and only if there exists a vertex-surjective homomorphism $G\to K_\ell$.
    \item[(2)] A \textit{$\ell$-clique} or \textit{complete $\ell$-partite graph} $K_{m_1,\ldots,m_\ell}$ has $\mid V_1\mid=m_1,\ldots,\mid V_\ell\mid=m_\ell$, and the edge set is defined as $E=\{\{u,v\}:~u\in V_i, v\in V_j \text{ for any $i\neq j$}\}$. A $2$-clique is also called a \textit{biclique}. Note that $K_\ell=K_{1,\ldots,1}$, where $m_1=\cdots=m_\ell=1$.  
\end{enumerate} 
\end{definition}


\section{(Injective) hom-complexity}\label{sec:one}
In this section, we introduce the notion of (injective) hom-complexity and its properties. Several examples are provided to support this theory. 

\subsection{Definitions and Examples} 
Given two graphs $G$ and $H$, in general, a homomorphism $f:G\to H$ may not exist. A significant challenge in graph theory is identifying such homomorphisms. Therefore, we present the main definition of this work.

\begin{definition}[(Injective) hom-complexity]\label{defn:complexity}
Let $G$ and $H$ be graphs.
    \begin{enumerate}
        \item[(1)]  The \textit{hom-complexity from $G$ to $H$}, denoted by $\text{C}(G;H)$, is the least positive integer $k$ such that there exist subgraphs $G_1,\ldots,G_k$ of $G$ satisfying  $G=G_1\cup\cdots\cup G_k$, with the property that for each $G_i$, there exists a homomorphism $f_i:G_i\to H$. We set $\text{C}(G;H)=\infty$ if no such integer $k$ exists. This invariant, in view of Proposition~\ref{prop:hom-cov-number}, can also be called the \textit{$H$-covering number} of $G$. 
        \item[(2)]  The \textit{injective hom-complexity from $G$ to $H$}, denoted by $\text{IC}(G;H)$, is the least positive integer $k$ such that there exist subgraphs $G_1,\ldots,G_k$ of $G$ satisfying $G=G_1\cup\cdots\cup G_k$, and for each $G_i$, there exists an injective homomorphism $f_i:G_i\to H$. We set $\text{IC}(G;H)=\infty$ if no such integer $k$ exists.
    \end{enumerate} 
\end{definition}

A collection $\mathcal{M}=\{f_i:G_i\to H\}_{i=1}^\ell$, where $G_1,\ldots,G_\ell$ are subgraphs of $G$ such that $G=G_1\cup\cdots\cup G_\ell$ and  each $f_i:G_i\to H$ is a homomorphism, is called a \textit{quasi-homomorphism} from $G$ to $H$. A quasi-homomorphism $\mathcal{M}=\{f_i:G_i\to H\}_{i=1}^\ell$ is termed \textit{optimal} if $\ell=\text{C}(G;H)$. Observe that a unitary quasi-homomorphism $\{f:G\to H\}$ is optimal and constitutes a homomorphism from $G$ to $H$. Additionally, any quasi-homomorphism $\mathcal{M}=\{f_i:G_i\to H\}_{i=1}^\ell$ induces a map $f:V(G)\to V(H)$ defined by  $f(v)=f_i(v)$, where $i$ is the least index such that $v\in V(G_i)$. Likewise, a collection $\mathcal{M}=\{f_i:G_i\to H\}_{i=1}^\ell$, where  $G_1,\ldots,G_\ell$ are subgraphs of $G$ such that $G=G_1\cup\cdots\cup G_\ell$ and each $f_i:G_i\to H$ is an injective homomorphism, is called an \textit{injective quasi-homomorphism} from $G$ to $H$. An injective quasi-homomorphism $\mathcal{M}=\{f_i:G_i\to H\}_{i=1}^\ell$ is termed \textit{optimal} if $\ell=\text{IC}(G;H)$.

\medskip Note that if we remove or add isolated vertices from $G$, its hom-complexity $\text{C}(G;H)$ does not change. This statement does not hold for the injective hom-complexity. For example, $\mathrm{IC}(K_2\sqcup \{\ast\};K_2)=2$, whereas $\mathrm{IC}(K_2;K_2)=1$.

\medskip By Definition~\ref{defn:complexity}, we can also make the following remark. 

\begin{remark}\label{rem:def-ob}
\noindent \begin{itemize}
    \item[(1)]  $\text{C}(G;H)\leq \text{IC}(G;H)$ for any graphs $G$ and $H$, since any injective quasi-homomorphism is a quasi-homomorphism. 
    \item[(2)]  $\text{C}(G;H)=1$ if and only if there exists a homomorphism $G\to H$ (i.e., $G$ is $H$-colourable). Additionally, $\text{IC}(G;H)=1$  if and only if there exists an injective homomorphism $G\to H$, which is equivalent to saying that $H$ admits a copy of $G$ as a subgraph.
    \item[(3)] The hom-complexity $\text{C}(G;H)$ coincides with the least positive integer $k$ such that there exist subgraphs $G_1,\ldots,G_k$ of $G$ satisfying  $G=G_1\cup\cdots\cup G_k$, and each $G_i$ is $H$-colourable. 
    \item[(4)] Since $H$-colorability does not depend on isolated vertices, we have that $\text{C}(G;H)$ coincides with the least positive integer $k$ such that there exist spanning subgraphs $G_1,\ldots,G_k$ of $G$ satisfying $G=G_1\cup\cdots\cup G_k$, and each $G_i$ is $H$-colourable. 
    \item[(5)] By Item (4), $\text{C}(G;H)$ coincides with the least positive integer $k$ such that there exists a homomorphism \[G\to H^{(k)},\] where $H^{(k)}$ is the graph whose vertex set is the Cartesian product $V(H)^k=V(H)\times\cdots\times V(H)$ (taken $k$ times), and where two $k$-tuples $x=(x_1,\ldots,x_k)$ and $y=(y_1,\ldots,y_k)$ are adjacent if and only if $x_i$ and $y_i$ are adjacent in $H$ for at least one index $i\in\{1,\ldots,k\}$\footnote{Note that $H^{(k)}$ and $H^{\boxtimes k}$ are not the same graph in general. Here, $H^{\boxtimes k}$ denotes the \textit{strong graph product}, whose vertex set is also $V(H)^k$, and where two $k$-tuples $x=(x_1,\ldots,x_k)$ and $y=(y_1,\ldots,y_k)$ are adjacent if and only if, for each $i\in\{1,\ldots,k\}$, either $x_i=y_i$ or $x_i$ and $y_i$ are adjacent in $H$; see \cite[p. 36]{hammack2011}). Clearly, $H^{\boxtimes k}$ is a subgraph of $H^{(k)}$, and they coincide  when $H=K_n$.}.  
\end{itemize}
\end{remark}

We do not use the following result, but we present it as a general statement. By \cite[Proposition 1.10, p. 8]{hell2004} and Remark~\ref{rem:def-ob}(2), we can infer the following statement.

\begin{lemma}\label{prop:complexity-one}
    Let $G$ and $H$ be graphs. We have the following:
    \begin{enumerate}
        \item[(1)]  $\mathrm{C}(G;H)=1$ if and only if the vertices of $G$ can be partitioned into sets $S_x$, $x\in V(H)$, such that the following property is satisfied: If $xy\notin E(H)$, then $uv\notin E(G)$ whenever $u\in S_x$ and $v\in S_y$.
        \item[(2)]  $\mathrm{IC}(G;H)=1$ if and only if the vertices of $G$ can be partitioned into unitary sets $S_x$, $x\in V(H)$, such that the following property is satisfied: If $xy\notin E(H)$, then $uv\notin E(G)$ whenever $u\in S_x$ and $v\in S_y$.
    \end{enumerate}
\end{lemma}

The following example demonstrates that the hom-complexity can be strictly less than the injective hom-complexity.

\begin{example}\label{exam:complexity-less-injective}
    Let $G$ and $H$ be directed graphs defined as follows: $V(G)=\{a,b,c\}$, $E(G)=\{(a,b),(c,b)\}$, $V(H)=\{a',b',c'\}$ and $E(H)=\{(c',a'),(c',b')\}$. 
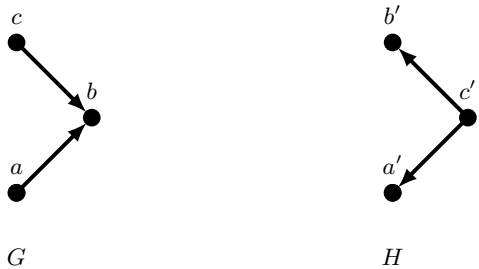
    $$
\begin{tikzpicture}
\Vertex[x=0,y=-1,size=0.2,label=$a$,position=above,color=black]{A} 
\Vertex[x=1, size=0.2,label=$b$,position=above,color=black]{B}
\Vertex[x=0, y=1, size=0.2,label=$c$,position=above,color=black]{C} 
\Vertex[x=5, y=-1, size=0.2,label=$a'$,position=above,color=black]{D}
\Vertex[x=6, size=0.2,label=$c'$,position=above,color=black]{E} 
\Vertex[x=5,y=1,size=0.2,label=$b'$,position=above,color=black]{F} 
\Vertex[x=0,y=-1,size=0.2,distance=0.5cm,label=$G$,position=below,color=black]{M} 
\Vertex[x=5, y=-1, size=0.2,distance=0.5cm,label=$H$,position=below,color=black]{N}
\Edge[Direct,color=black](A)(B) 
\Edge[Direct,color=black](C)(B)
\Edge[Direct,color=black](E)(D)
\Edge[Direct,color=black](E)(F)
  \end{tikzpicture}
 $$ Note that the map $f:V(G)\to V(H)$ defined by $f(a)=f(c)=c'$ and  $f(b)=b'$ is a homomorphism from $G$ to $H$. Hence, we have $\mathrm{C}(G;H)=1$.
 
On the other hand, observe that $\mathrm{indeg}(b)=2$. We will show  that $\mathrm{IC}(G;H)= 2$. Suppose there exists an injective homomorphism $f:G\to H$. Then, we have $\mathrm{indeg}(f(b))\geq 2$ (see Lemma~\ref{lem:hom-degree}). This leads to a contradiction, because in $H$, we have: \begin{align*}
   \mathrm{indeg}(a')&= 1,\\
   \mathrm{indeg}(b')&= 1,\\
   \mathrm{indeg}(c')&= 0.\\
 \end{align*} So, we conclude that $\mathrm{IC}(G;H)\geq 2$. Additionally, consider the subgraphs $G_1$ and $G_2$ of $G$, defined as follows:\begin{align*}
     V(G_1)&=\{a,b\},\\
     E(G_1)&=\{(a,b)\},\\
     V(G_2)&=\{b,c\},\\
     E(G_2)&=\{(c,b)\}.\\
 \end{align*} Together with the injective homomorphisms $f_1:G_1\to H$ and $f_2:G_2\to H$, defined by: \begin{align*}
     f_1(a)&=c',\\
     f_1(b)&=a',\\
     f_2(b)&=b',\\
     f_2(c)&=c'.\\
 \end{align*} Note that $G=G_1\cup G_2$. Thus, we have $\mathrm{IC}(G,H)\leq 2$. Therefore, we conclude that $\mathrm{IC}(G,H)=2$. 
\end{example}

Next example demonstrates that the hom-complexity can be equal to the injective hom-complexity.

\begin{example}\label{exam:complexity-equal-inyective}
    Consider the directed graphs $G$ and $H$ defined as follows: $V(G)=\{a,b,c\}$, $E(G)=\{(a,b),(b,c)\}$, and $H$ as in Example~\ref{exam:complexity-less-injective}. 
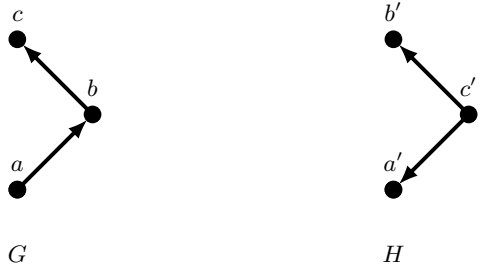
    $$
\begin{tikzpicture}
\Vertex[x=0,y=-1,size=0.2,label=$a$,position=above,color=black]{A} 
\Vertex[x=1, size=0.2,label=$b$,position=above,color=black]{B}
\Vertex[x=0, y=1, size=0.2,label=$c$,position=above,color=black]{C} 
\Vertex[x=5, y=-1, size=0.2,label=$a'$,position=above,color=black]{D}
\Vertex[x=6, size=0.2,label=$c'$,position=above,color=black]{E} 
\Vertex[x=5,y=1,size=0.2,label=$b'$,position=above,color=black]{F} 
\Vertex[x=0,y=-1,size=0.2,distance=0.5cm,label=$G$,position=below,color=black]{M} 
\Vertex[x=5, y=-1, size=0.2,distance=0.5cm,label=$H$,position=below,color=black]{N}
\Edge[Direct,color=black](A)(B) 
\Edge[Direct,color=black](B)(C)
\Edge[Direct,color=black](E)(D)
\Edge[Direct,color=black](E)(F)
  \end{tikzpicture}
 $$ Observe that $\mathrm{indeg}(b)=1$ and $\mathrm{outdeg}(b)=1$. We will show that $\mathrm{IC}(G;H)=\mathrm{C}(G;H)= 2$. Suppose there exists  a homomorphism $f:G\to H$. Note that $f$ must be injective. Then, we have $\mathrm{indeg}(f(b))\geq 1$ and $\mathrm{outdeg}(f(b))\geq 1$ (see Lemma~\ref{lem:hom-degree}). This leads to a contradiction because, in $H$, we have: \begin{align*}
   \mathrm{indeg}(a')&= 1,\\
   \mathrm{outdeg}(a')&= 0,\\
   \mathrm{indeg}(b')&= 1,\\
   \mathrm{outdeg}(b')&= 0,\\
   \mathrm{indeg}(c')&= 0,\\
   \mathrm{outdeg}(c')&= 2.\\
 \end{align*} Thus, we conclude that $\mathrm{C}(G;H)\geq 2$. Next, consider the subgraphs $G_1$ and $G_2$ of $G$, defined as follows:\begin{align*}
     V(G_1)&=\{a,b\},\\
     E(G_1)&=\{(a,b)\},\\
     V(G_2)&=\{b,c\},\\
     E(G_2)&=\{(b,c)\}.\\
 \end{align*} Together with the injective homomorphisms $f_1:G_1\to H$ and $f_2:G_2\to H$, defined by: \begin{align*}
     f_1(a)&=c',\\
     f_1(b)&=a',\\
     f_2(b)&=c',\\
     f_2(c)&=b'.\\
 \end{align*} Note that $G=G_1\cup G_2$. Thus, we have $\mathrm{IC}(G;H)\leq 2$. Therefore, we conclude that $\mathrm{IC}(G;H)=\mathrm{C}(G;H)=2$. 
\end{example}

Now, we will present an example involving undirected graphs.

\begin{example}\label{exam:complexity-equal-inyective-no-dirigido}
    Let $G=K_3$ and $H$ be an undirected graph defined as follows:  \begin{align*}
       V(H)&=\{a',b',c'\},\\
       E(H)&=\{\{c',a'\},\{c',b'\}\}.\\
    \end{align*}
    $$
\begin{tikzpicture}
\Vertex[x=-1,y=0,size=0.2,label=$1$,position=above,color=black]{A} 
\Vertex[x=1, size=0.2,label=$2$,position=above,color=black]{B}
\Vertex[x=0, y=1, size=0.2,label=$3$,position=above,color=black]{C} 
\Vertex[x=5, y=-1, size=0.2,label=$a'$,position=above,color=black]{D}
\Vertex[x=6, size=0.2,label=$c'$,position=above,color=black]{E} 
\Vertex[x=5,y=1,size=0.2,label=$b'$,position=above,color=black]{F} 
\Vertex[x=0,y=1,size=0.2,distance=1.5cm,label=$K_3$,position=below,color=black]{M} 
\Vertex[x=5, y=-1, size=0.2,distance=0.5cm,label=$H$,position=below,color=black]{N}
\Edge[color=black](A)(B) 
\Edge[color=black](C)(B)
\Edge[color=black](C)(A)
\Edge[color=black](E)(D)
\Edge[color=black](E)(F)
  \end{tikzpicture}
 $$ We will show that $\mathrm{IC}(K_3;H)=\mathrm{C}(K_3;H)=2$. First, if there exists a homomorphism $K_3\to H$, then we must have $\chi(K_3)\leq \chi(H)$. This leads to a contradiction, since $\chi(K_3)=3$ and $\chi(H)=2$. Therefore, we conclude that $\mathrm{C}(K_3;H)\geq 2$. 
 
 On the the hand, consider the subgraphs $G_1$ and $G_2$ of $K_3$, defined as follows:\begin{align*}
     V(G_1)&=\{1,b,3\},\\
     E(G_1)&=\{\{1,3\},\{2,3\}\},\\
     V(G_2)&=\{1,2\},\\
     E(G_2)&=\{\{1,2\}\}.\\
 \end{align*} Now, define the injective homomorphisms $f_1:G_1\to H$ and $f_2:G_2\to H$ as follows: \begin{align*}
     f_1(1)&=a',\\
     f_1(2)&=b',\\
     f_1(3)&=c',\\
     f_2(1)&=a',\\
     f_2(2)&=c'.\\
 \end{align*} Note that $G=G_1\cup G_2$. Thus, we have $\mathrm{IC}(K_3;H)\leq 2$. Therefore, we conclude that $\mathrm{IC}(K_3;H)=\mathrm{C}(K_3;H)=2$. 
\end{example}

\subsection{Triangular Inequality} Given a homomorphism $f:G\to H$ and a subgraph $K$ of $H$, the \textit{image inverse} of $K$ through $f$ is the graph $f^{-1}(K)$, defined as follows:
 \begin{itemize}
     \item The vertex set is given by $V(f^{-1}(K))=f^{-1}(V(K))$.
     \item Two vertices $u$ and $v$ in $f^{-1}(K)$ are adjacent if and only if they are adjacent in $G$ and $f(u)$ and $f(v)$ are adjacent in $K$. 
 \end{itemize} Note that the restriction map $f_|:f^{-1}(V(K))\to V(K)$ is a vertex-surjective homomorphism from the graph $f^{-1}(K)$ to $K$, called the \textit{restriction homomorphism}, and is denoted by $f_|:f^{-1}(K)\to K$. 

 \medskip Given three graphs $G,H$, and $K$, there is a relation between the hom-complexities $\mathrm{C}(G;H), \mathrm{C}(H;K)$, and $\mathrm{C}(G;K)$. Likewise, the same holds for the injective hom-complexity. 

\begin{theorem}[Triangular Inequality]\label{thm:inequality-three-graphs}
  Let $G,H$, and $K$ be graphs. Then, \[\mathrm{C}(G;K)\leq \mathrm{C}(G;H)\cdot \mathrm{C}(H;K) \quad \text{ and } \quad \mathrm{IC}(G;K)\leq \mathrm{IC}(G;H)\cdot \mathrm{IC}(H;K).\]  
\end{theorem}
\begin{proof}
  Let $m=\mathrm{C}(G;H)$ and $n=\mathrm{C}(H;K)$. Let $\mathcal{M}_1=\{g_{i}:G_{i}\to H\}_{i=1}^{m}$ be an optimal quasi-homomorphism from $G$ to $H$, and $\mathcal{M}_2=\{h_{j}:H_{j}\to K\}_{j=1}^{n}$ be an optimal quasi-homomorphism from $H$ to $K$. Define  $G_{i,j}:=g_{i}^{-1}(H_j)$ for each $i\in\{1,\ldots,m\}$ and each $j\in\{1,\ldots,n\}$ (noting that some $G_{i,j}$ may be empty). We have $G=\bigcup_{i,j=1}^{m,n} G_{i,j}$. If $G_{i,j}\neq\varnothing$, it is a subgraph of $G_i$ (and consequently a subgraph of $G$). We  also consider the restriction homomorphism $(g_i)_|:G_{i,j}\to H_j$. This leads to the composition \[G_{i,j}\stackrel{(g_i)_|}{\to} H_j\stackrel{h_j}{\to} K.\] Therefore, we obtain $\mathrm{C}(G;K)\leq m\cdot n=\mathrm{C}(G;H)\cdot \mathrm{C}(H;K)$.  

  Likewise, we obtain the inequality $\mathrm{IC}(G;K)\leq \mathrm{IC}(G;H)\cdot \mathrm{IC}(H;K)$ because if $g_{i}$ and $h_{j}$ are injective, then the composition $G_{i,j}\stackrel{(g_i)_|}{\to} H_j\stackrel{h_j}{\to} K$ is also injective. 
\end{proof}

The inequality in Theorem~\ref{thm:inequality-three-graphs} is sharp. For instance, consider $K=H$; then $\mathrm{C}(G;H)= \mathrm{C}(G;H)\cdot \mathrm{C}(H;H)$. In addition, see Example~\ref{exam:chi-chi-2} below.


\subsection{Graph invariant} The next result establishes a close connection between the hom-complexities and homomorphisms. 

\begin{theorem}\label{prop:complexity-subgraphs}
    Let $G'\to G$ and $H'\to H$ be homomorphisms.  
    \begin{itemize}
        \item[(1)] We have: \[\mathrm{C}(G';H)\leq \mathrm{C}(G;H)\leq \mathrm{C}(G;H').\] 
        \item[(2)] Moreover, if $G'\to G$ and $H'\to H$ are injective, then \[\mathrm{IC}(G';H)\leq \mathrm{IC}(G;H)\leq \mathrm{IC}(G;H').\] 
    \end{itemize}
\end{theorem}
\begin{proof}
 It is a direct consequence of Theorem~\ref{thm:inequality-three-graphs}.
\end{proof}

From Theorem~\ref{prop:complexity-subgraphs}(1), we observe that if $G'\to G$ and $G\to G'$, then $\mathrm{C}(G';H)=\mathrm{C}(G;H)$ for any graph $H$. Similarly, if $H'\to H$ and $H\to H'$, then $\mathrm{C}(G;H')=\mathrm{C}(G;H)$ for any graph $G$. In particular, this shows  that (injective) hom-complexity is a graph invariant, meaning it is preserved under graph isomorphisms.

\begin{corollary}[Graph Invariant]\label{cor:invariant-iso-complexity}
    If $G'$ is isomorphic to $G$ and $H'$ is isomorphic to $H$, then \[\mathrm{C}(G;H)=\mathrm{C}(G';H') \quad \text{ and } \quad \mathrm{IC}(G;H)=\mathrm{IC}(G';H').\]
\end{corollary}

Furthermore, Theorem~\ref{prop:complexity-subgraphs} can be used to show the nonexistence of homomorphisms. 

\begin{proposition}\label{prop:no-existencia}
\noindent\begin{enumerate}
    \item[(1)]  Let $G$ and $G'$ be graphs. We have:
    \begin{enumerate}
        \item[(i)]  If $\mathrm{C}(G';H)>\mathrm{C}(G;H)$ for some graph $H$, then $G'\not\rightarrow G$.
         \item[(ii)]  If $\mathrm{IC}(G';H)>\mathrm{IC}(G;H)$ for some graph $H$, then there is no injective homomorphism from $G'$ to  $G$. 
    \end{enumerate}
     \item[(2)]  Let $H$ and $H'$ be graphs. We have:
    \begin{enumerate}
        \item[(i)]  If $\mathrm{C}(G;H)>\mathrm{C}(G;H')$ for some graph $G$, then $H'\not\rightarrow H$. 
         \item[(ii)] If $\mathrm{IC}(G;H)>\mathrm{IC}(G;H')$ for some graph $G$, then there is no injective homomorphism from $H'$ to  $H$. 
    \end{enumerate}
\end{enumerate}    
\end{proposition}
\begin{proof}
   It is sufficient to use the contrapositive of each implication in Theorem~\ref{prop:complexity-subgraphs}.
\end{proof}

Hence, we have the following statement, which provides a lower bound for the chromatic number.

\begin{corollary}\label{cor:complejidad-chro}
    If $\mathrm{C}(G;H)>\mathrm{C}(K_n;H)$ or $\mathrm{C}(H;G)<\mathrm{C}(H;K_n)$ for some graph $H$, then $\chi(G)\geq n+1$. Alternatively, if $\chi(G)\leq n$, then $\mathrm{C}(G;K)\leq \mathrm{C}(K_n;K)$ and $\mathrm{C}(K;G)\geq \mathrm{C}(K;K_n)$ for any graph $K$.
\end{corollary}
\begin{proof}
    By Proposition~\ref{prop:no-existencia}(1)(i) and (2)(i), we have  $G\not\rightarrow K_n$. Therefore, it follows that $\chi(G)\geq n+1$. 
\end{proof}

Corollary~\ref{cor:complejidad-chro} motivates us to compute the hom-complexities $\mathrm{C}(K_n;H)$ and $\mathrm{C}(H;K_n)$. These are given in Example~\ref{exem:complexity-kj-k2}.

\medskip Let $H$ be a subgraph of $G$. A \textit{retraction} of $G$ to $H$ is a homomorphism $r:G\to H$ such that $r(v)=v$ for all $v\in V(H)$. If there is a retraction of $G$ to $H$, we say that $G$ \textit{retracts} to $H$, and that $H$ is a \textit{retract} of $G$. A \textit{core} is a graph which does not retract to a proper subgraph. For example, $K_n$ is a core. Each graph $G$ has a unique (up to isomorphism) retract $G'$ that is a core. We say that $G'$ is \textit{the core} of $G$ (cf. \cite[p. 19]{hell2004}). The core of a bipartite graph with at least one edge is $K_2$. The core of a $k$-clique is $K_k$. The core of a graph with loops is a single loop. 

\medskip Moreover, Theorem~\ref{prop:complexity-subgraphs} implies that hom-complexity can be restricted to cores. 

\begin{proposition}\label{prop:hom-complexity-core}
    Let $G$ and $H$ be graphs, and let $G'$ and $H'$ be their respective cores. We have
    \begin{enumerate}
        \item[(1)] $\mathrm{C}(G;H)=\mathrm{C}(G';H)=\mathrm{C}(G';H')=\mathrm{C}(G;H')$.
        \item[(2)] $\mathrm{IC}(G';H)\leq\min\{\mathrm{IC}(G;H),\mathrm{IC}(G';H')\}\leq\max\{\mathrm{IC}(G;H),\mathrm{IC}(G';H')\}\leq\mathrm{IC}(G;H')$.
    \end{enumerate} 
\end{proposition}
\begin{proof}
    It follows from Theorem~\ref{prop:complexity-subgraphs}, because there are homomorphisms $G'\hookrightarrow G$, $H'\hookrightarrow H$ (the inclusions) and $G\to G'$, $H\to H'$ (retractions). 
\end{proof}

Proposition~\ref{prop:hom-complexity-core}(1) implies the following example.

\begin{example}\label{exam:complexity-h-n-clique}
Let $G$ be any graph. \begin{enumerate}
     \item[(1)] If $H$ is a bipartite graph with at least one edge, then  \[\mathrm{C}(G;H)=\mathrm{C}(G;K_2)\qquad \text{ and } \qquad \mathrm{C}(H;G)=\mathrm{C}(K_2;G).\]
     \item[(2)]  If $K_{m_1,\ldots,m_n}$ is a $n$-clique, then the equalities \[\mathrm{C}(G;K_{m_1,\ldots,m_n})=\mathrm{C}(G;K_n) \qquad \text{ and } \qquad \mathrm{C}(K_{m_1,\ldots,m_n};G)=\mathrm{C}(K_n;G)\] always hold. 
 \end{enumerate}  
\end{example}

Note that if $G$ is a graph with loops and $H$ is a graph without loops, then $\mathrm{C}(G;H)=\infty$. For instance, $\mathrm{C}(G;K_n)=\infty$ for any graph with loops $G$ and any $n\geq 1$.

\medskip Observe that $\mathrm{IC}(K_2;G)=1$ if and only if $E(G)\neq\varnothing$. Additionally, for the injective hom-complexity $\mathrm{IC}(G;K_2)$, we have: 

\begin{theorem}\label{thm:injective-category-compute}
    Let $G$ be a graph without isolated vertices and with $E(G)\neq\varnothing$. Then, we have \[\mathrm{IC}(G;K_2)=\begin{cases}
        \infty,&\hbox{ if $G$ has loops;}\\
        |E(G)|,&\hbox{ if $G$ does not have loops.}\\
    \end{cases}\] 
\end{theorem}
\begin{proof}
The case where $G$ has loops is straightforward. Now, suppose that $G$ does not have loops. Observe that $\mathrm{IC}(G;K_2)\leq|E(G)|$, because for each edge $e=uv\in E(G)$, we can consider the subgraph $L_e=L$ given by $V(L)=\{u,v\}$ and $E(L)=\{uv\}$. Take the injective homomomorphism $f_e=f:L\to K_2$ defined by $f(u)=1$ and $f(v)=2$. Then, the collection$\{f_e:L_e\to K_2\}_{e\in E(G)}$ forms an injective quasi-homomorphism from $G$ to $K_2$. 

Now, if $\{f_\ell:G_\ell\to K_2\}_{\ell=1}^{m}$ is an optimal injective quasi-homomorphism from $G$ to $K_2$, then each $G_\ell$ satisfies $|V(G_\ell)|=2$ and $|E(G_\ell)|=1$, i.e., for each $G_\ell$ we have an edge of $G$ (this is due to the fact that $G$ has no isolated vertices). Thus, we have $m\geq |E(G)|$, leading to $\mathrm{IC}(G;K_2)\geq |E(G)|$. Therefore, we conclude that $\mathrm{IC}(G;K_2)= |E(G)|$.
\end{proof}

Theorem~\ref{thm:injective-category-compute} implies the following example.

\begin{example}\label{exam:ic-clique-biclique}
  \noindent\begin{enumerate}
      \item[(1)] For any $n\geq 2$, it follows that \[\mathrm{IC}(K_n;K_2)=\dfrac{n(n-1)}{2}.\]
      \item[(2)] For any $p,q\geq 1$, we have \[\mathrm{IC}(K_{p,q};K_2)=pq.\]
  \end{enumerate}  
\end{example}

\begin{remark}
    Given a graph $G$ with $E(G)=\varnothing$, let $m=\dfrac{|V(G)|}{2}$. In this case, we note that the equality $\mathrm{IC}(G;K_2)=\lceil m\rceil$ holds.
\end{remark}

\subsection{Sub-additivity} The following statement demonstrates the sub-additivity property of (injective) hom-complexity. 

\begin{theorem}[Sub-additivity]\label{thm:category-union}
    Let $G,H$ be graphs, and let $A,B$ be subgraphs of $G$ such that $G=A\cup B$. Then: \begin{itemize}
        \item[(1)] $\max\{\mathrm{C}(A;H),\mathrm{C}(B;H)\}\leq \mathrm{C}(G;H)\leq \mathrm{C}(A;H)+\mathrm{C}(B;H).$ 
         \item[(2)] $\max\{\mathrm{IC}(A;H),\mathrm{IC}(B;H)\}\leq \mathrm{IC}(G;H)\leq \mathrm{IC}(A;H)+\mathrm{IC}(B;H).$ 
    \end{itemize} 
\end{theorem}
\begin{proof}
\noindent \begin{itemize}
        \item[(1)] The inequality $\max\{\mathrm{C}(A;H),\mathrm{C}(B;H)\}\leq \mathrm{C}(G;H)$ follows from Theorem~\ref{prop:complexity-subgraphs}(1), applied to the inclusions $A\hookrightarrow G$ and $B\hookrightarrow G$. To demonstrate the other inequality, suppose that $\mathrm{C}(A;H)=m$ and $\mathrm{C}(B;H)=k$. Let $\{f_i:F_i\to H\}_{i=1}^{m}$ be an optimal quasi-homomorphism from $A$ to $H$,  and $\{g_j:G_j\to H\}_{j=1}^{k}$ be an optimal quasi-homomorphism from $B$ to $H$. The combined collection $\{f_1:F_1\to H,\ldots,f_m:F_m\to H,g_1:G_1\to H,\ldots,g_k:G_k\to H\}$ forms a quasi-homomorphism from $G$ to $H$. Therefore, we have $\mathrm{C}(G;H)\leq m+k=\mathrm{C}(A;H)+\mathrm{C}(B;H)$. 

        This completes the proof of the sub-additivity of hom-complexity.

        \item[(2)] The inequality $\max\{\mathrm{IC}(A;H),\mathrm{IC}(B;H)\}\leq \mathrm{IC}(G;H)$ follows from Theorem~\ref{prop:complexity-subgraphs}(2), applied to the inclusions $A\hookrightarrow G$ and $B\hookrightarrow G$. To establish the other inequality, suppose that $\mathrm{IC}(A;H)=m$ and $\mathrm{IC}(B;H)=k$. Let $\{f_i:F_i\to H\}_{i=1}^{m}$ be an optimal injective quasi-homomorphism from $A$ to $H$ and $\{g_j:G_j\to H\}_{j=1}^{k}$ be an optimal injective quasi-homomorphism from $B$ to $H$. Then, the combined collection $\{f_1:F_1\to H,\ldots,f_m:F_m\to H,g_1:G_1\to H,\ldots,g_k:G_k\to H\}$ is an injective quasi-homomorphism from $G$ to $H$. Consequently, we have $\mathrm{IC}(G;H)\leq m+k=\mathrm{IC}(A;H)+\mathrm{IC}(B;H)$.

        This completes the proof of the sub-additivity of injective hom-complexity.
    \end{itemize}  
\end{proof}

Theorem~\ref{thm:category-union} implies the following corollary:

\begin{corollary}\label{cor:linear-complexity}
    Let $G$ and $H$ be graphs, and $A$ and $T$ be subgraphs of $G$ such that $G=A\cup T$. Then:
    \begin{enumerate}
        \item[(1)] If $\mathrm{C}(T;H)=1$, then \[\mathrm{C}(A;H)\leq \mathrm{C}(G;H)\leq \mathrm{C}(A;H)+1.\] 
        \item[(2)] If $\mathrm{IC}(T;H)=1$, then \[\mathrm{IC}(A;H)\leq \mathrm{IC}(G;H)\leq \mathrm{IC}(A;H)+1.\] 
    \end{enumerate}
\end{corollary}

\medskip The \textit{path} $P_m$ is the graph with vertices $1,2,\ldots,m$ and edges $12,23,\ldots,(m-1)m$.

\begin{example}\label{exam:cicle-k2}
    Consider the cycle $C_{m}$, where $m\geq 3$, given by \begin{align*}
        V(C_m)&=\{1,2,\ldots,m\},\\
        E(C_m)&=\{12,23,\ldots,(m-1)m,m1\}.\\
    \end{align*} Note that $C_m=P_m\cup K_2$ and $\mathrm{C}(P_m;K_2)=1$, $C(K_2;K_2)=1$. Here, $V(K_2)=\{1,m\}$ and $E(K_2)=\{m1\}$. Therefore, $\mathrm{C}(C_m;K_2)\leq 2$ (see Corollary~\ref{cor:linear-complexity}(1)). We will show that: \[\mathrm{C}(C_m;K_2)=\begin{cases}
        1,&\hbox{ if $m$ is even;}\\
         2,&\hbox{ if $m$ is odd.}\\
    \end{cases}\] Indeed, in the case where $m=2k$, consider the map $f:V(C_{2k})\to K_2$ defined by: \[
    f(i)=\begin{cases}
        1,&\hbox{ for $i$ odd and $1\leq i \leq 2k-1$;}\\
        m,&\hbox{ for $i$ even and $2\leq i \leq 2k$.}\\
    \end{cases}
    \] Note that $f$ is a homomorphism from $C_{2k}$ to $K_2$ and thus $\mathrm{C}(C_{2k};K_2)=1$. 

    On the other hand, since $\chi(C_m)=3$ for any odd $m\geq 3$, we have $\mathrm{C}(C_{m};K_2)\geq 2$. Therefore, $\mathrm{C}(C_{m};K_2)=2$ for any odd $m\geq 3$. 
\end{example}

The following result shows that the first inequality of Theorem~\ref{thm:category-union}(1) can be an equality. 

\begin{proposition}\label{prop:complexity-disjoint-union}
  Let $G$ be a graph, and let $A$ and $B$ be subgraphs of $G$ such that $V(A)\cap V(B)=\varnothing$ and $G=A\sqcup B$ (see (\ref{eq:disjoint-union})). Then, for any graph $H$, we have \[\mathrm{C}(G;H)=\max\{\mathrm{C}(A;H),\mathrm{C}(B,H)\}.\]   
\end{proposition}
\begin{proof}
 The inequality $\max\{\mathrm{C}(A;H),\mathrm{C}(B,H)\}\leq \mathrm{C}(A\sqcup B;H)$ follows from Theorem~\ref{thm:category-union}(1). We will now verify the inequality $\mathrm{C}(A\sqcup B;H)\leq m$, where $m=\max\{\mathrm{C}(A;H),\mathrm{C}(B,H)\}$. In fact, let $\{f_i:A_i\to H\}_{i=1}^{m}$ and $\{g_i:B_i\to H\}_{i=1}^{m}$ be quasi-homomorphisms from $A$ to $H$ and from $B$ to $H$, respectively. Note that the collection $\{f_i\sqcup g_i:A_i\sqcup B_i\to H\}_{i=1}^{m}$ forms a quasi-homomorphism from $G$ to $H$. Therefore, we have $\mathrm{C}(G;H)\leq m=\max\{\mathrm{C}(A;H),\mathrm{C}(B,H)\}$.
\end{proof}

Observe that the condition $V(A)\cap V(B)=\varnothing$ in  Proposition~\ref{prop:complexity-disjoint-union} cannot be removed. To illustrate this, consider the graphs $A$ and $B$ defined as follows:  $V(A)=\{1,2,3\}$, $E(A)=\{12,13\}$, $V(B)=\{2,3\}$ and $E(B)=\{23\}$. Note that, $\mathrm{C}(A;K_2)=1$ and $\mathrm{C}(B;K_2)=1$. However, since $A\cup B=K_3$, we have $\mathrm{C}(A\cup B;K_2)=2$. 

\subsection{Product inequality}
Given two graphs $G_1$ and $G_2$, the \textit{graph product} $G_1\times G_2$ is defined by $V(G_1\times G_2)=V(G_1)\times V(G_2)$, and two vertices $u=(u_1,u_2)$ and $v=(v_1,v_2)$ in $V(G_1\times G_2)$ are adjacent if $u_1v_1\in E(G_1)$ and $u_2v_2\in E(G_2)$ (see \cite[p. 37]{hell2004}). Given two homomorphisms $f_1:G_1\to H_1$ and  $f_2:G_2\to H_2$, their \textit{product} $f_1\times f_2:V(G_1)\times V(G_2)\to V(H_1)\times V(H_2)$ is defined as $(f_1\times f_2)(u_1,u_2)=(f_1(u_1),f_2(u_2))$. This forms a homomorphism from $G_1\times G_2$ to $H_1\times H_2$. Note that if $A$ is a subgraph of $G_1$ and $B$ is a subgraph of $G_2$, then $A\times B$ is a subgraph of $G_1\times G_2$. Furthermore, each coordinate projection $p_j:G_1\times G_2\to G_j$ defined by $p_j(u_1,u_2)=u_j$ for $j=1,2$ is a homomorphism. In addition, the \textit{diagonal homomorphism} $D_G:G\to G\times G$ given by $D_G(u)=(u,u)$ for all $u\in V(G)$ is injective. Hence, we have  $\mathrm{C}(G_1\times G_2;G_j)=1$ for each $j=1,2$, and $\mathrm{C}(G;G\times G)=\mathrm{IC}(G;G\times G)=1$.

\medskip We have the following statement.

\begin{proposition}[Diagonal Product Property]\label{prop:diagonal}
    Let $G$ and $H$ be graphs. The following hold: \begin{enumerate}
        \item[(1)]  $\mathrm{C}(G;H\times H)=\mathrm{C}(G;H)=\mathrm{C}(G\times G;H)$.
         \item[(2)]  $\mathrm{IC}(G;H\times H)\leq \mathrm{IC}(G;H)\leq \mathrm{IC}(G\times G;H)$. 
    \end{enumerate}
\end{proposition}
\begin{proof}
    \noindent\begin{enumerate}
        \item[(1)] This follows from Theorem~\ref{prop:complexity-subgraphs}(1), applied to one of the projections along with the diagonal homomorphism.
        \item[(2)] This follows from Theorem~\ref{prop:complexity-subgraphs}(2), applied to the diagonal homomorphism.
    \end{enumerate}
\end{proof}

A direct consequence of Proposition~\ref{prop:diagonal}(1) is the following corollary. 

\begin{corollary}\label{cor:complexity-product-product}
    Let $G$ and $H$ be graphs. Then, we have \[\mathrm{C}(G\times G;H\times H)=\mathrm{C}(G;H).
    \]
\end{corollary}
\begin{proof}
   From Proposition~\ref{prop:diagonal}(1), we have:\begin{align*}
        \mathrm{C}(G\times G;H\times H)&=\mathrm{C}(G\times G;H)\\
        &=\mathrm{C}(G;H).
    \end{align*}
\end{proof}

The following statement presents the product inequality. 

\begin{theorem}[Product Inequality]\label{thm:product-inequality-complexity}
    Let $G_1$, $G_2$, $H_1$ and $H_2$ be graphs. Then, we have:
    \begin{enumerate}
        \item[(1)]  $\max\{\mathrm{C}(G_1\times G_2;H_1),\mathrm{C}(G_1\times G_2;H_2)\}\leq \mathrm{C}(G_1\times G_2;H_1\times H_2)\leq \min\{\mathrm{C}(G_1;H_1)\cdot \mathrm{C}(G_2;H_2),\mathrm{C}(G_1;H_1\times H_2),\mathrm{C}(G_2;H_1\times H_2)\}.$
         \item[(2)]  $\mathrm{IC}(G_1\times G_2;H_1\times H_2)\leq \mathrm{IC}(G_1;H_1)\cdot \mathrm{IC}(G_2;H_2).$ 
    \end{enumerate} 
\end{theorem}
\begin{proof}
 \noindent\begin{enumerate}
     \item[(1)] Set $m=\mathrm{C}(G_1;H_1)$, $n=\mathrm{C}(G_2;H_2)$,  and let $\mathcal{M}_1=\{f_{i,1}:G_{i,1}\to H_1\}_{i=1}^{m}$ and  $\mathcal{M}_2=\{f_{j,2}:G_{j,2}\to H_2\}_{j=1}^{n}$ be optimal quasi-homomorphisms from $G_1$ to $H_1$ and from $G_2$ to $H_2$, respectively. The collection $\mathcal{M}_1\times \mathcal{M}_2=\{f_{i,1}\times f_{j,2}:G_{i,1}\times G_{j,2}\to H_1\times H_2\}_{i=1,j=1}^{m,n}$ is a quasi-homomorphism from $G_1\times G_2$ to $H_1\times H_2$. Thus, we have $\mathrm{C}(G_1\times G_2;H_1\times H_2)\leq m\cdot n=\mathrm{C}(G_1;H_1)\cdot \mathrm{C}(G_2;H_2)$. 

    The other inequalities follow from Theorem~\ref{prop:complexity-subgraphs}(1), applied to each projection $G_1\times G_2\to G_j$ and $H_1\times H_2\to H_j$. 
     \item[(2)] Let $m=\mathrm{IC}(G_1;H_1)$, $n=\mathrm{iC}(G_2;H_2)$, and let $\mathcal{M}_1=\{f_{i,1}:G_{i,1}\to H_1\}_{i=1}^{m}$ and  $\mathcal{M}_2=\{f_{j,2}:G_{j,2}\to H_2\}_{j=1}^{n}$ be optimal injective quasi-homomorphisms from $G_1$ to $H_1$ and from $G_2$ to $H_2$, respectively. The collection $\mathcal{M}_1\times \mathcal{M}_2=\{f_{i,1}\times f_{j,2}:G_{i,1}\times G_{j,2}\to H_1\times H_2\}_{i=1,j=1}^{m,n}$ is an injective quasi-homomorphism from $G_1\times G_2$ to $H_1\times H_2$. Thus, we have $\mathrm{IC}(G_1\times G_2;H_1\times H_2)\leq m\cdot n=\mathrm{IC}(G_1;H_1)\cdot \mathrm{IC}(G_2;H_2)$. 
 \end{enumerate}   
\end{proof}

Theorem~\ref{thm:product-inequality-complexity}(1), in conjunction  with Proposition~\ref{prop:diagonal}(1), implies the following statement.

\begin{proposition}\label{prop:G-H-H}
 Let $G$, $H_1$, and $H_2$ be graphs. Then, the following inequalities hold: \[\max\{\mathrm{C}(G;H_1),\mathrm{C}(G;H_2)\}\leq\mathrm{C}(G;H_1\times H_2)\leq\mathrm{C}(G;H_1)\cdot\mathrm{C}(G;H_2).\]   
\end{proposition}

Hence, we obtain the following example.

\begin{example}\label{exam:com-1-other}
    Let $G$, $H_1$, and $H_2$ be graphs. Suppose that $\{i,j\}=\{1,2\}$. Proposition~\ref{prop:G-H-H} implies that if $\mathrm{C}(G;H_j)=1$,  then $\mathrm{C}(G;H_1\times H_2)=\mathrm{C}(G;H_i)$. 
\end{example}


\subsection{Lower bound} Let $G$ be a simple graph, and  let $G_1,\ldots,G_m$ be subgraphs of $G$ such that $G=G_1\cup\cdots\cup G_m$. Then, we have $\chi(G)\leq \prod_{j=1}^m\chi(G_j)$ (see Proposition~\ref{prop:chromatic-union}). Hence, we obtain the following lower bound for hom-complexity.

\begin{theorem}[Lower Bound]\label{thm:lower-bound}
 Let $G$ and $H$ be simple graphs. The inequality \[\chi(G)\leq\left(\chi(H)\right)^{\mathrm{C}(G;H)}\] holds. Equivalently, $\log_{\chi(H)}\chi(G)\leq \mathrm{C}(G;H)$.    
\end{theorem}
\begin{proof}
 If $\mathrm{C}(G;H)=\infty$, then the desired inequality holds. Now, suppose that $m=\mathrm{C}(G;H)<\infty$ and consider $G_1,\ldots,G_m$ subgraphs of $G$ such that $G=G_1\cup\cdots\cup G_m$, with a homomorphism $G_j\to H$ for each $G_j$ (thus $\chi(G_j)\leq\chi(H)$). Then, we have: \begin{align*}
        \chi(G)&\leq \prod_{j=1}^m\chi(G_j)\\
        &\leq \prod_{j=1}^m\chi(H)\\
        &=\left(\chi(H)\right)^m.
    \end{align*} 

    Here is an alternative argument, which strongly relies on Remark~\ref{rem:def-ob}(5). Suppose that $m=\mathrm{C}(G;H)<\infty$; then there exists a homomorphism $G\to H^{(m)}$. Since $\chi( H^{(m)})\leq \left(\chi(H)\right)^m$, it follows that $\chi(G)\leq \left(\chi(H)\right)^{\mathrm{C}(G;H)}$.
\end{proof}

Theorem~\ref{thm:lower-bound} implies the following statement.

\begin{corollary}\label{cor:chro-lower-bound}
Let $m\geq 1$ be an integer. Let $G$ and $H$ be simple graphs such that $\chi(G)\geq 2$. If $\left(\chi(H)\right)^{m-1}+1\leq \chi(G)$,  then $\mathrm{C}(G;H)\geq m$. Thus, we have \[\max\{m:~\left(\chi(H)\right)^{m-1}+1\leq \chi(G)\}\leq \mathrm{C}(G;H).\]  
\end{corollary}
\begin{proof}
    Suppose that $\mathrm{C}(G;H)\leq m-1$ (with $m\geq 2$). By Theorem~\ref{thm:lower-bound}, we have \begin{align*}
    \chi(G)&\leq \left(\chi(H)\right)^{\mathrm{C}(G;H)}\\
        &\leq\left(\chi(H)\right)^{m-1}.
    \end{align*} This leads to a contradiction, as $\left(\chi(H)\right)^{m-1}+1\leq \chi(G)$.
\end{proof}


\subsection{Upper bound} Note that $\mathrm{C}(K_{m_1,\ldots,m_k};K_k)=1$, since there exists a  homomorphism $f:K_{m_1,\ldots,m_k}\to K_k$ defined by $f(v)=i$ whenever $v\in V_i$. 

\medskip The following result provides an upper bound for $\mathrm{C}(K_{j};K_i)$.

\begin{proposition}\label{prop:upper-bound}
  Let $m\geq 1$ be an integer. The inequality $\mathrm{C}(K_{j};K_i)\leq m$ holds for any $j\leq i^{m}$.    
\end{proposition}
\begin{proof}
We will apply induction on $m\geq 1$. For $m=1$, the statement holds immediately. Suppose the statement is valid for $k$, with $1\leq k\leq m$. We will check that it holds for $m+1$; that is, for $j\leq i^{m+1}$, we will verify that the inequality $\mathrm{C}(K_{j};K_i)\leq m+1$ is fulfilled. 

If $j\leq i$, then $\mathrm{C}(K_{j};K_i)=1\leq m+1$. Now, assume $j> i$. Note that \[K_{j}=\begin{cases}
        \left(\bigsqcup_{i \text{ times}}^{}K_\ell\right)\cup K_{\overset{\ell}{i}},& \hbox{ if $j=\ell i$;}\\  & \\
        \left(\bigsqcup_{r \text{ times}}^{}K_\ell\sqcup \bigsqcup_{i-r \text{ times}}^{}K_{\ell-1}\right)\cup K_{\overset{\ell}{r},\overset{\ell-1}{i-r}},& \hbox{ if $j=(\ell-1)i+r$ with $1\leq r<i$;}\\
    \end{cases}\] where $K_{\overset{\ell}{i}}:=K_{\underbrace{\ell,\ldots,\ell}_{i \text{ times}}}$ and $K_{\overset{\ell}{r},\overset{\ell-1}{i-r}}:=K_{\underbrace{\ell,\ldots,\ell}_{r \text{ times}},\underbrace{\ell-1,\ldots,\ell-1}_{i-r \text{ times}}}$. Since $j>i$, we have $\ell\geq 2$. From Theorem~\ref{thm:category-union}(1) together with Proposition~\ref{prop:complexity-disjoint-union}, we have that  \[\mathrm{C}(K_j;K_i)\leq \mathrm{C}(K_\ell;K_i)+1.\] Note that $j\leq i^{m+1}$ implies that $\ell\leq i^{m}$. Then, by the induction  hypothesis, $\mathrm{C}(K_\ell;K_i)\leq m$. Therefore, $\mathrm{C}(K_j;K_i)\leq m+1$ for any $j\leq i^{m+1}$.    
\end{proof}

\begin{remark}\label{rem:grotzch}
    Note that the statement $\mathrm{C}(K_{j};H)\leq m$ for any graph $H$ with $j\leq\left(\chi(H)\right)^m$ is not necessarily true. For example, consider $H=M_4$ as the Gr\"{o}tzch graph, which has 11 vertices and 20 edges, and is the smallest triangle-free graph that cannot be 3-colored. In this case, we have $j=3$, $\chi(M_4)=4$, and $\mathrm{C}(K_{3};M_4)=\mathrm{IC}(K_{3};M_4)=2$.
   \end{remark} 

Proposition~\ref{prop:upper-bound} implies an upper bound for $\mathrm{C}(G;H)$. Given a graph $H$, let $\omega(H)=\max\{j:~K_j\text{ is a subgraph of } H\}$, which is called the \textit{clique number} of $H$. We will consider $\omega(H)<\infty$. Note that $\omega(H)\leq \chi(H)$ for any simple graph $H$.   

\begin{theorem}[Upper Bound]\label{thm:general-upper-bound}
 Let $n\geq 1$ be an integer, and let $G$ be a simple graph and $H$ be any graph. If $\chi(G)\leq \left(\omega(H)\right)^n$, then $\mathrm{C}(G;H)\leq n$. Thus, we have \[\mathrm{C}(G;H)\leq\min\{n:~\chi(G)\leq \left(\omega(H)\right)^n\}.\] 
 \end{theorem}
\begin{proof}
    Set $G\to K_j$, where $j=\chi(G)$. Also, let $i=\omega(H)$, so that $K_i$ is a subgraph of $H$. From Theorem~\ref{prop:complexity-subgraphs}(1), we have \begin{align*}
      \mathrm{C}(G;H) &\leq \mathrm{C}(K_j;H)\\
      &\leq \mathrm{C}(K_j;K_i).
    \end{align*} Since $\mathrm{C}(K_j;K_i)\leq n$ (see Proposition~\ref{prop:upper-bound}), it follows that $\mathrm{C}(G;H)\leq n$.     
\end{proof}

Corollary~\ref{cor:chro-lower-bound} together with Theorem~\ref{thm:general-upper-bound} implies the following statement. 

\begin{proposition}\label{cor:chro-alpha-complexity}
Let $G$ and $H$ be simple graphs such that $\chi(G),\omega(H)\geq 2$. Then \[\max\{m:~\left(\chi(H)\right)^{m-1}+1\leq \chi(G)\}\leq\mathrm{C}(G;H)\leq \min\{n:~\chi(G)\leq \left(\omega(H)\right)^n\}.\]  
\end{proposition}

Note that $\chi(H)=2$ implies that $\omega(H)=2$. Hence, Proposition~\ref{cor:chro-alpha-complexity} implies the following example.

\begin{example}\label{exam:codomain-chromatic-2}
Let $m\geq 1$ be an integer. Let $G$ and $H$ be simple graphs such that $\chi(G)\geq 2$ and $\chi(H)=2$. If $2^{m-1}+1\leq \chi(G)\leq 2^m$, then $\mathrm{C}(G;H)= m$. Thus, we have \[\mathrm{C}(G;H)=\min\{k:~\chi(G)\leq 2^k\}=\lceil \log_2\chi(G)\rceil.\] Of course, the equalities \[\mathrm{C}(G;H)=\max\{m:~2^{m-1}+1\leq \chi(G)\}=\lfloor \log_{2}\left(2\left(\chi(G)-1\right)\right)\rfloor\] also hold.
\end{example}

Furthermore, we have the following example.

\begin{example}\label{exam:chi-chi-2}
    Let $G, H$, and $K$ be simple graphs such that $\chi(H)=\chi(K)=2$. By Example~\ref{exam:codomain-chromatic-2}, we have $\mathrm{C}(H;K)=1$ and \[\mathrm{C}(G;K)= \mathrm{C}(G;H)=\min\{k:~\chi(G)\leq 2^k\}=\lceil\log_2\chi(G)\rceil.\]
\end{example}

Moreover, by Proposition~\ref{cor:chro-alpha-complexity}, we have the following formula.

\begin{corollary}\label{cor:omega-equal-chi}
Let $G, H$ be simple graphs such that $\chi(G)\geq 2$ and $\omega(H)=\chi(H)\geq 2$. Then,  \[\mathrm{C}(G;H)=\min\{k:~\chi(G)\leq \left(\chi(H)\right)^k\}=\lceil\log_{\chi(H)}\chi(G) \rceil.\] Of course, the equalities \[\mathrm{C}(G;H)=\max\{m:~\left(\chi(H)\right)^{m-1}+1\leq \chi(G)\}=\lfloor \log_{\chi(H)}\left(\chi(H)\left(\chi(G)-1\right)\right)\rfloor\] also hold.
\end{corollary}

Observe that we cannot relax the condition $\omega(H)=\chi(H)$. For instance, consider $H=M_4$ as the Gr\"{o}tzch graph (see Remark~\ref{rem:grotzch}). In this case, $\omega(M_4)=2$, $\chi(M_4)=4$, $\lceil\log_{4}3 \rceil=1$ and $\mathrm{C}(K_3;M_4)=2$. 

\medskip Corollary~\ref{cor:omega-equal-chi} implies the following example.

\begin{example}\label{exem:complexity-kj-k2}
 For any simple graph $G$ such that $\chi(G)\geq 2$ and any $n\geq 2$, we have: \[\mathrm{C}(G;K_n)=\min\{k:~\chi(G)\leq n^{k}\}=\lceil\log_{n}\chi(G)\rceil.\] Note that $\omega(K_n)=\chi(K_n)=n$. For instance: \[\mathrm{C}(K_{j};K_i)=\min\{k:~j\leq i^{k}\}=\lceil\log_{i}j\rceil\] for any $j,i\geq 2$. Furthermore, the equalities \[\mathrm{C}(K_n;G)=\min\{k:~n\leq \left(\chi(G)\right)^{k}\}=\lceil\log_{\chi(G)}n\rceil\] hold whenever $\omega(G)=\chi(G)$.
  \end{example}

\medskip From Example~\ref{exem:complexity-kj-k2}, we know that $\mathrm{C}(G;K_\ell)=k$ whenever $\ell^{k-1}<\chi(G)\leq \ell^{k}$. Furthermore, by the Proof of Theorem~\ref{thm:general-upper-bound}, we have an optimal quasi-homomorphism $\{f_i:G_i\to K_\ell\}_{i=1}^{k}$ from $G$ to $K_\ell$, where each $\hat{G}_i$ is $K_\ell$-colourable. The following statement (which will be fundamental to Theorem~\ref{prop:biparticity-complexity-k2}(1)) shows that we can have an optimal quasi-homomorphism $\{\hat{f}_i:\hat{G}_i\to K_\ell\}_{i=1}^{k}$ from $G$ to $K_\ell$, where each $\hat{G}_i$ is a $\ell$-partite subgraph of $G$. This is nontrivial, since being \aspas{$K_\ell$-colourable} is not the same as being \aspas{$\ell$-partite}, see  Definition~\ref{defn:l-partite}(1). 

\begin{theorem}\label{prop:covering-l-partite}
Let $\ell\geq 2$. 
\begin{enumerate}
\item[(1)] Let $k\geq 2$. If $\ell^{k-1}\leq j\leq \ell^k$, then there exist subgraphs $G_1,\ldots,G_{k}$ of $K_{j}$ such that \[K_{j}=G_1\cup\cdots\cup G_{k}\] and each $G_i$ is $\ell$-partite (i.e., there exists a vertex-surjective homomorphism $f_i:G_i\to K_\ell$). Hence, for $\ell^{k-1}< j\leq \ell^k$, there exists an optimal quasi-homomorphism $\{f_i:G_i\to K_\ell\}_{i=1}^{k}$ from $K_{j}$ to $K_\ell$ where each $f_i:G_i\to K_\ell$ is a vertex-surjective homomorphism.
    \item[(2)] Let $G$ be a simple graph and $k\geq 2$. If $\ell^{k-1}<\chi(G)\leq \ell^k$, then there exists an optimal quasi-homomorphism $\{\hat{f}_i:\hat{G}_i\to K_\ell\}_{i=1}^{k}$ from $G$ to $K_\ell$ where each $\hat{f}_i:\hat{G}_i\to K_\ell$ is a vertex-surjective homomorphism (i.e., each $\hat{G}_i$ is a $\ell$-partite subgraph of $G$). 
\end{enumerate}
\end{theorem}
\begin{proof}
    \noindent\begin{enumerate}
        \item[(1)] We will apply induction on $k\geq 2$. For $k=2$. For the case $j=\ell$, consider $G_1=G_2=K_\ell$. Note that each $G_i$ is $\ell$-partite. Thus, we trivially obtain $\ell$-partite subgraphs $G_1$ and $G_2$ of $K_{\ell}$ such that $K_{\ell}=G_1\cup G_2$.  
        
        For $\ell<j\leq \ell^2$, note that \[K_{j}=\begin{cases}
          \left(K^1_{\ell}\sqcup\cdots\sqcup K^\ell_{\ell}\right)\cup K_{\underbrace{\ell,\ldots,\ell}_{\ell \text{ times}}},&\hbox{ if $j=\ell^2$;}\\ \left(\bigsqcup_{t=1}^rK^t_m\sqcup\bigsqcup_{s=1}^{\ell-r}K^s_{m-1}\right)\cup K_{\overset{m}{r},\overset{m-1}{\ell-r}},&\hbox{ if $j=(m-1)\ell+r$} 
        \end{cases} \] for some integers $m\geq 2$ (using $\ell<j$) and $1\leq r<\ell$. Here \[K_{\overset{m}{r},\overset{m-1}{\ell-r}}:=K_{\underbrace{m,\ldots,m}_{r \text{ times}},\underbrace{m-1,\ldots,m-1}_{\ell-r \text{ times}}},\] and each $K^a_{b}$ is a copy of $K_b$. 

        For the case $j=\ell^2$, consider $G_1=K^1_{\ell}\sqcup\cdots\sqcup K^\ell_{\ell}$ and $G_2=K_{\underbrace{\ell,\ldots,\ell}_{\ell \text{ times}}}$. Note that each $G_i$ is $\ell$-partite. Thus, we obtain $\ell$-partite subgraphs $G_1$ and $G_2$ of $K_{\ell^2}$ such that $K_{\ell^2}=G_1\cup G_2$.

       For the case $j=(m-1)\ell+r$. Since $\ell<j\leq \ell^2$, we have $1\leq m-1< m\leq \ell$. Consider $G_1=\bigsqcup_{t=1}^rK^t_m\sqcup\bigsqcup_{s=1}^{\ell-r}K^s_{m-1}$ and $G_2=K_{\overset{m}{r},\overset{m-1}{\ell-r}}$. Note that $G_1$ and $G_2$ are $\ell$-partite. Thus, in this case, we obtain $\ell$-partite subgraphs $G_1$ and $G_2$ of $K_{j}$ such that $K_{j}=G_1\cup G_2$.  

        Now, suppose that the result holds for $p$, with $2\leq p\leq k$. We will check that it holds for $k+1$; that is, for $\ell^{k}\leq j\leq \ell^{k+1}$, we will verify that there exist $\ell$-partite subgraphs $G_1,\ldots,G_{k+1}$ of $K_{j}$ such that $K_{j}=G_1\cup\cdots\cup G_{k+1}$. In fact, for the case $j=\ell^{k}$, by the induction hypothesis, there are $\ell$-partite subgraphs $G_1,\ldots,G_k$ of $K_{\ell^k}$ such that $K_{\ell^k}=G_1\cup\cdots\cup G_k$. Consider $G_{k+1}=G_k$. Hence, we have $\ell$-partite subgraphs $G_1,\ldots,G_k,G_{k+1}$ of $K_{\ell^k}$ such that $K_{\ell^k}=G_1\cup\cdots\cup G_k\cup G_{k+1}$.  
        
        For the case $j=\ell^{k+1}$, note that \[K_{\ell^{k+1}}=\left(K^1_{\ell^k}\sqcup\cdots\sqcup K^\ell_{\ell^k}\right)\cup K_{\underbrace{\ell^k,\ldots,\ell^k}_{\ell \text{ times}}},\] where each $K^i_{\ell^k}$ is a copy of $K_{\ell^k}$. By the induction  hypothesis, for each $i$, there are $\ell$-partite subgraphs $G^i_1,\ldots,G^i_k$ of $K^i_{\ell^k}$ such that $K^i_{\ell^k}=G^i_1\cup\cdots\cup G^i_k$. For each $\nu=1,\ldots,k$, consider \begin{align*}
            G_\nu&=G^1_\nu\sqcup\cdots\sqcup G^\ell_\nu,\\
        \end{align*} and $G_{k+1}=K_{\underbrace{\ell^k,\ldots,\ell^k}_{\ell \text{ times}}}$. Note that $G_1,\ldots,G_k,G_{k+1}$ are $\ell$-partite subgraphs of $K_{\ell^{k+1}}$. Furthermore, since $G_1\cup\cdots\cup G_k=K^1_{\ell^k}\sqcup\cdots\sqcup K^\ell_{\ell^k}$, we have that $K_{\ell^{k+1}}=G_1\cup\cdots\cup G_k\cup G_{k+1}$. 

        Therefore, we obtain $\ell$-partite subgraphs $G_1,\ldots,G_{k+1}$ of $K_{\ell^{k+1}}$ such that $K_{\ell^{k+1}}=G_1\cup\cdots\cup G_{k+1}$.

        For the case $\ell^{k}<j< \ell^{k+1}$, we have $j=(m-1)\ell+r$ for some integers $m\geq 2$ and $1\leq r<\ell$. Note that \[K_j=\left(\bigsqcup_{t=1}^rK^t_m\sqcup\bigsqcup_{s=1}^{\ell-r}K^s_{m-1}\right)\cup K_{\underbrace{m,\ldots,m}_{r \text{ times}},\underbrace{m-1,\ldots,m-1}_{\ell-r \text{ times}}},\] where each $K^t_m$ is a copy of $K_m$ and each $K^s_{m-1}$ is a copy of $K_{m-1}$. 

        Since $\ell^{k}<j< \ell^{k+1}$, we have $\ell^{k-1}\leq m-1< m\leq \ell^{k}$. By the induction hypothesis, for each $t$, there are $\ell$-partite subgraphs $G^t_1,\ldots,G^t_k$ of $K^t_{m}$ such that $K^t_{m}=G^t_1\cup\cdots\cup G^t_k$. Moreover, for each $s$, there are $\ell$-partite subgraphs $G^s_1,\ldots,G^s_k$ of $K^s_{m-1}$ such that $K^s_{m}=G^s_1\cup\cdots\cup G^s_k$. For each $\nu=1,\ldots,k$, consider \begin{align*}
            G_\nu&=\bigsqcup_{t=1}^rG^t_\nu\sqcup\bigsqcup_{s=1}^{\ell-r} G^s_\nu,
        \end{align*} and $G_{k+1}=K_{\underbrace{m,\ldots,m}_{r \text{ times}},\underbrace{m-1,\ldots,m-1}_{\ell-r \text{ times}}}$. Note that $G_1,\ldots,G_k,G_{k+1}$ are $\ell$-partite subgraphs of $K_{j}$. Furthermore, since $G_1\cup\cdots\cup G_k=\bigsqcup_{t=1}^rK^t_m\sqcup\bigsqcup_{s=1}^{\ell-r}K^s_{m-1}$, we have that $K_{j}=G_1\cup\cdots\cup G_k\cup G_{k+1}$. 

        Therefore, we obtain $\ell$-partite subgraphs $G_1,\ldots,G_{k+1}$ of $K_{j}$ such that $K_{j}=G_1\cup\cdots\cup G_{k+1}$.
       
        \item[(2)] Let $j:=\chi(G)$ and consider a (vertex-surjective) homomorphism $f:G\to K_j$. Since $\ell^{k-1}<j\leq \ell^k$, by Item (1), there exists an optimal quasi-homomorphism $\{f_i:G_i\to K_\ell\}_{i=1}^{k}$ from $K_j$ to $K_\ell$ where each $f_i:G_i\to K_\ell$ is a vertex-surjective homomorphism. 
        
        By the first part of the proof of Theorem~\ref{prop:complexity-subgraphs}, the homomorphism $f:G\to K_j$ induces a quasi-homomorphism $\{\hat{f}_i:\hat{G}_i\to K_\ell\}_{i=1}^{k}$ from $G$ to $K_\ell$ where each $\hat{G}_i=f^{-1}(G_i)$ and $\hat{f}_i=f_i \circ f_{|}$. Recall that $f_{|}:f^{-1}(G_i)\to G_i$ denotes the restriction homomorphism for each $i$. Since $f_{|}$ and $f_i$ are vertex-surjective homomorphisms, each $\hat{f}_i:\hat{G}_i\to K_\ell$ is a vertex-surjective homomorphism. Note that this quasi-homomorphism $\{\hat{f}_i:\hat{G}_i\to K_\ell\}_{i=1}^{k}$ is optimal because $\mathrm{C}(G;K_\ell)=k$.

        Therefore, we obtain an optimal quasi-homomorphism $\{\hat{f}_i:\hat{G}_i\to K_\ell\}_{i=1}^{k}$ from $G$ to $K_\ell$ where each $\hat{f}_i:\hat{G}_i\to K_\ell$ is a vertex-surjective homomorphism. 
    \end{enumerate}
\end{proof}

\section{Covering Number}\label{sec:cov-number} 

In this section, we discuss a connection between the (injective) hom-complexity and some well-known covering numbers. 

\medskip A family of subgraphs $\{G_\lambda\}_{\Lambda}$ of $G$ is called an \textit{edge covering} of $G$ if $E(G)=\bigcup_{\lambda\in\Lambda} E(G_\lambda)$. Let $\mathcal{S}$ be a family of subgraphs of $G$. The \textit{$\mathcal{S}$-covering number} of $G$, denoted by $\sigma_{\mathcal{S}}(G)$, is the least positive integer $k$ such that there exists an edge covering with $k$ elements $\{G_1,\ldots,G_k\}$ of $G$ such that each $G_i\in\mathcal{S}$. We set $\sigma_{\mathcal{S}}(G)=\infty$ if no such integer $k$ exists (cf.  \cite{schwartz2022}). For example, by definition, $\sigma_{\varnothing}(G)=\infty$.

\medskip Recall that $G=G_1\cup\cdots\cup G_k$ means $V(G)=V(G_1)\cup\cdots\cup V(G_k)$ and $E(G)=E(G_1)\cup\cdots\cup E(G_k)$. In particular, such a family $\{G_1,\ldots,G_k\}$ is an edge covering of $G$. The converse holds in the case that $G$ has not isolated vertices, as shown in the following remark.

\begin{remark}\label{rem:edge-cov-without-isolated-v-cov}
 Let $G_1,\ldots, G_k$ be subgraphs of $G$ such that $E(G)=E(G_1)\cup\cdots\cup E(G_k)$. If $G$ does not contain isolated vertices, then $V(G)=V(G_1)\cup\cdots\cup V(G_k)$. Indeed, let $v\in V(G)$. Since $G$ has not isolated vertices, there exists $u\in V(G)$ such that $uv\in E(G)$, and then $uv\in E(G_i)$ for some $i$. Thus, $v\in V(G_i)$ for some $i$. Therefore, $V(G)=V(G_1)\cup\cdots\cup V(G_k)$.    
\end{remark}

Hence, we have the following statement, which says that the hom-complexity coincides with a covering number.

\begin{proposition}[Hom-complexity is a covering number]\label{prop:hom-cov-number}
  Let $G$ and $H$ be graphs. Let $\mathcal{H}$ be the family of all subgraphs of $G$ which are $H$-colourable. We have \[\mathrm{C}(G;H)= \sigma_{\mathcal{H}}(G).\]  
\end{proposition}
\begin{proof}
    By Remark~\ref{rem:def-ob}(3), the inequality $\mathrm{C}(G;H)\geq\sigma_{\mathcal{H}}(G)$ always holds. To show the other inequality, first note that the covering number $\sigma_{\mathcal{H}}(G)$ and the hom-complexity $\mathrm{C}(G;H)$ do not change if we remove isolated vertices from $G$. Hence, we can suppose that $G$ has no isolated vertices. Set $k=\sigma_{\mathcal{H}}(G)$ and consider $G_1,\ldots,G_k$ subgraphs of $G$ such that $E(G)=E(G_1)\cup\cdots\cup E(G_k)$, and each $G_i$ is $H$-colourable. By Remark~\ref{rem:edge-cov-without-isolated-v-cov}, we have $V(G)=V(G_1)\cup\cdots\cup V(G_k)$. Therefore, $\mathrm{C}(G;H)\leq k=\sigma_{\mathcal{H}}(G)$.   
\end{proof}

In particular, Proposition~\ref{prop:hom-cov-number} shows that the hom-complexity $\text{C}(G;H)$ coincides with the least positive integer $k$ such that there exist subgraphs $G_1,\ldots,G_k$ of $G$ satisfying $E(G)=E(G_1)\cup\cdots\cup E(G_k)$, and each $G_i$ is $H$-colourable. Also, note that $\sigma_{\mathcal{H}}(G)$ coincides with the least positive integer $k$ such that there exist subgraphs $G_1,\ldots,G_k$ of $G$ satisfying $E(G)=E(G_1)\cup\cdots\cup E(G_k)$, $E(G_i)\cap E(G_j)=\varnothing$ for any $i\neq j$ (i.e., the sets $G_1,\ldots,G_k$ are pairwise disjoint), and each $G_i$ is $H$-colourable.   

\medskip We have the following remark.

\begin{remark}\label{rem:in-contrast-ic}
Let $\text{inj-}\mathcal{H}$ be the family of all subgraphs $L$ of $G$ such that there exists an injective homomorphism of $L$ to $H$. Note that, $\mathrm{IC}(G;H)\geq\sigma_{\text{inj-}\mathcal{H}}(G)$. However, Proposition~\ref{prop:hom-cov-number} does not hold for injective hom-complexity. For example, for $G=K_2\sqcup\{\ast\}$ and $H=K_2$, we have $\mathrm{IC}(G;H)=2$ and $\sigma_{\text{inj-}\mathcal{H}}(G)=1$. This example shows that $\mathrm{IC}(G;H)$ depends of the isolated points from $G$. In contrast, the covering number $\sigma_{\text{inj-}\mathcal{H}}(G)$ does not change if isolated vertices are removed from $G$.   
\end{remark} 

Recall that \[\omega(G)=\max\{j:~\text{$G$ admits a clique with $j$ vertices}\}\] denotes the clique number of $G$. In particular, for any clique $L$ of $G$, we have $\chi(L)=|V(L)|\leq \omega(G)$, and hence, there exists an injective homomorphism $L\to K_{\omega(G)}$.

\medskip Also, we recall several examples of covering numbers.

\begin{definition}[Clique covering number, $\ell$-particity, and $\ell$-partite dimension]\label{defn:cc-l-particity-l-partite-dim}
 Let $G$ be a simple graph.
 \begin{enumerate}
 \item[(1)] The \textit{clique covering number} $cc(G)$ of $G$ is the least positive integer $k$ such that there exists an edge covering with $k$ elements $\{G_1,\ldots,G_k\}$ of $G$ such that each $G_i$ is a clique. We set $cc(G)=\infty$ if no such integer $k$ exists (cf. \cite{erdos1966}, \cite{schwartz2022},  \cite{gregory1982}, \cite{schwartz2022}). Note that $cc(G)$ does not change if isolated vertices are removed from $G$. Furthermore, the inequality $cc(G)\leq |E(G)|$ always hold.   
  \item[(2)] For $\ell\geq 2$, the \textit{$\ell$-particity} $\beta_{\ell}(G)$ of $G$ is the least positive integer $k$ such that there exists an edge covering with $k$ elements $\{G_1,\ldots,G_k\}$ of $G$ such that each $G_i$ is a $\ell$-partite graph. We set $\beta_{\ell}(G)=\infty$ if no such integer $k$ exists (cf. \cite[p. 132]{harary1977}). The case $\ell=2$ was studied in \cite[p. 203]{harary1970} and \cite[p. 131]{harary1977}, and it is called the \textit{biparticity} of $G$, and denoted by $\beta(G)$.
     \item[(3)] For $\ell\geq 2$, the \textit{$\ell$-partite dimension} $\text{d}_{\ell}(G)$ of $G$ is the least positive integer $k$ such that there exists an edge covering with $k$ elements $\{G_1,\ldots,G_k\}$ of $G$ such that each $G_i$ is a $\ell$-clique (i.e., a complete $\ell$-partite graph). We set $\text{d}_{\ell}(G)=\infty$ if no such integer $k$ exists. The case $\ell=2$ was studied in \cite[p. 128]{fishburn1996}, and it is called the \textit{bipartite dimension} of $G$, and denoted by $\text{d}(G)$. 
 \end{enumerate} 
\end{definition}

 Let $\ell\geq 2$ and $G$ be a simple graph. Recall that $G$ is $\ell$-partite if and only if there exists a vertex-surjective homomorphism $G\to K_\ell$. We have the following statement.

\begin{proposition}\label{prop:complexity-lower-bound-cn}
 Let $G$ be a simple graph and $\ell\geq 2$. We have:
 \begin{enumerate}
     \item[(1)] If $G$ has no isolated vertices, then \[cc(G)\geq \mathrm{IC}(G;K_{\omega(G)})\geq \mathrm{C}(G;K_{\omega(G)}).\]
     \item[(2)] $\text{d}_{\ell}(G)\geq \beta_{\ell}(G)\geq \mathrm{C}(G;K_{\ell})$.
 \end{enumerate}
\end{proposition}
\begin{proof}
    \noindent\begin{enumerate}
        \item[(1)] Suppose that $G$ has no isolated vertices. Set $cc(G)=k<\infty$ and consider $G_1,\ldots,G_k$ subgraphs of $G$ such that $E(G)=E(G_1)\cup\cdots\cup E(G_k)$, and each $G_i$ is a clique (and, of course, there exists an injective homomorphism $G_i\to K_{\omega(G)}$ for each $i$). By Remark~\ref{rem:edge-cov-without-isolated-v-cov}, we have that $V(G)=V(G_1)\cup\cdots\cup V(G_k)$. Therefore, $\mathrm{IC}(G;H)\leq k=cc(G)$. 

        The inequality $\mathrm{IC}(G;K_{\omega(G)})\geq \mathrm{C}(G;K_{\omega(G)})$ always holds. 
        \item[(2)] Since a graph $L$ is $\ell$-partite if and only if there exists a vertex-surjective homomorphism $L\to K_\ell$, by Proposition~\ref{prop:hom-cov-number}, we have $\beta_{\ell}(G)\geq \mathrm{C}(G;K_{\ell})$. The inequality $\text{d}_{\ell}(G)\geq \beta_{\ell}(G)$ always holds. 
    \end{enumerate}
\end{proof}

For instance, Proposition~\ref{prop:complexity-lower-bound-cn}(1) presents a lower bound for the clique covering number in terms of the injective hom-complexity. This bound improves, for some examples, the lower bound $\log_2(|V(G)|+1)$ presented in \cite{gyarfas1990}. In \cite[Theorem 2, p. 107]{erdos1966}, the authors present the upper bound $\lfloor n^2/4\rfloor$, where $n=|V(G)|$.  

\begin{example}\label{exam:cc-w2}
 Let $G$ be a simple graph without isolated vertices such that $\omega(G)=2$. By Proposition~\ref{prop:complexity-lower-bound-cn}(1) together with Theorem~\ref{thm:injective-category-compute}, we have \[cc(G)= \mathrm{IC}(G;2)=|E(G)|.\] Recall that the inequality $cc(G)\leq |E(G)|$ always hold. For instance:
 \begin{enumerate}
     \item[(1)] $cc(P_m)=m-1$ for any $m\geq 2$, and $cc(C_n)=n$ for any $n\geq 3$.
     \item[(2)] $cc(K_{p,q})= \mathrm{IC}(K_{p,q};2)=|E(K_{p,q})|=pq$ for any $p,q\geq 1$.
     \item[(3)] For the Gr\"{o}tzch graph $M_4$ we have $\omega(M_4)=2$, and thus $cc(M_4)= \mathrm{IC}(M_4;2)=|E(M_4)|=20$. On the other hand, observe that $|V(M_4)|=11$, $\log_2(12)\leq 4$ and $\lfloor 121/4\rfloor=30$.     
 \end{enumerate}
\end{example}

Also, Proposition~\ref{prop:complexity-lower-bound-cn}(1) together with Example~\ref{exem:complexity-kj-k2} implies a new lower bound for $cc(G)$. 

\begin{proposition}\label{prop:properties-cov-number}
  Let $G$ be a simple graph without isolated vertices such that $\omega(G)\geq 2$, then \[cc(G)\geq \lceil\log_{\omega(G)}\chi(G)\rceil.\]
\end{proposition}

This lower bound is sharp; for example, $cc(K_n)=1=\lceil\log_{n}n\rceil$.  

\medskip On the other hand, note that, in general, if $G\to K_\ell$ does not imply that $G$ is $\ell$-partite. For example, we have the inclusion $K_{\ell-1}\hookrightarrow K_\ell$ but 
$K_{\ell-1}$ is not $\ell$-partite because $|V(K_{\ell-1})|=\ell-1$. We observe the following remark.

\begin{remark}\label{rem:vertex-l-partite}
Let $\ell\geq 2$.\begin{enumerate}
    \item[(1)] For any simple graph $G$ with $E(G)\neq\varnothing$, $G$ is bipartite if and only if there exists a homomorphism $G\to K_2$ (and, of course, it is a vertex-surjective homomorphism).
    \item[(2)] Given a $\ell$-partite subgraph $L$ of $K_n$, there exists a $\ell$-clique $\widetilde{L}$ of $K_n$ such that $L\subseteq \widetilde{L}$.
\end{enumerate}      
\end{remark}

We have the following statement which says that the hom-complexity $\mathrm{C}(G;K_{\ell})$ coincides with the $\ell$-particity of $G$. In addition, the hom-complexity $\mathrm{C}(K_n;K_{2})$ coincides with the bipartite dimension of $K_n$.  

 \begin{theorem}\label{prop:biparticity-complexity-k2}
   Let $\ell\geq 2$. \begin{enumerate}
       \item[(1)] Let $G$ be a simple graph such that $\mathrm{C}(G;K_{\ell})\geq 2$. We have \[\beta_\ell(G)=\mathrm{C}(G;K_{\ell}).\]  
       \item[(2)]  Let $n\geq 2$. We have \[\mathrm{d}(K_n)=\beta(K_n)=\mathrm{C}(K_n;K_{2}).\] 
   \end{enumerate} 
 \end{theorem}
 \begin{proof}
 \noindent\begin{enumerate}
     \item[(1)]  Suppose that $\mathrm{C}(G;K_{\ell})=k<\infty$ (i.e., $\ell^{k-1}<\chi(G)\leq\ell^k$ with $k\geq 2$). By Theorem~\ref{prop:covering-l-partite}(2), we can consider $G_1,\ldots,G_k$ subgraphs of $G$ such that $G=G_1\cup\cdots\cup G_k$ and each $G_i$ is $\ell$-partite. Therefore, $\beta_\ell(G)\leq k=\mathrm{C}(G;K_{\ell})$. The other inequality $\beta_\ell(G)\geq \mathrm{C}(G;K_{\ell})$ always holds (see Proposition~\ref{prop:complexity-lower-bound-cn}(2)). Thus, we conclude that $\beta_\ell(G)=\mathrm{C}(G;K_{\ell})$.
     \item[(2)] Recall that the inequalities $\mathrm{d}(K_n)\geq \beta(K_n)\geq \mathrm{C}(K_n;K_{2})$ always hold (see Proposition~\ref{prop:complexity-lower-bound-cn}(2)). We will check that $\mathrm{C}(K_n;K_{2})\geq \mathrm{d}(K_n)$. Indeed, suppose that $\mathrm{C}(K_n;K_{2})=k<\infty$. Consider $G_1,\ldots,G_k$ subgraphs of $K_n$ such that $K_n=G_1\cup\cdots\cup G_k$ and there exists a homomorphism $G_i\to K_2$ for each $i$. If $k=1$, then $n=2$ and thus $\mathrm{d}(K_2)=1$. 
     
     Now, suppose that $k\geq 2$. Note that $E(G_i)\neq\varnothing$ for each $i$ (otherwise, $\mathrm{C}(K_n;K_{2})<k$). By Remark~\ref{rem:vertex-l-partite}(1), each $G_i$ is a bipartite subgraph of $K_n$. Then, by Remark~\ref{rem:vertex-l-partite}(2), there exists a biclique $\widetilde{G}_i$ of $K_n$ such that $G_i\subseteq \widetilde{G}_i$ for each $i$. Note that, $K_n=\widetilde{G}_1\cup\cdots\cup \widetilde{G}_k$. Therefore, $\mathrm{d}(K_n)\leq k=\mathrm{C}(K_n;K_{2})$.
 \end{enumerate}
    \end{proof}

\begin{remark}\label{rem:bipa-dim-dif-complex}
Note that the statement $\mathrm{d}(G)=\mathrm{C}(G;K_{2})$ does not hold for all $G$. For example, for $G=P_4$, the path with 4 vertices, we have $\mathrm{C}(P_4;K_{2})=1$ and $\mathrm{d}(P_4)=2$ (since $P_4$ is not a biclique).       
\end{remark} 

As a direct consequence of Theorem~\ref{prop:biparticity-complexity-k2}(1) together with Example~\ref{exem:complexity-kj-k2}, we recover the formula for $\beta_\ell(G)$ obtained by Harary-Hsu-Miller in \cite[p. 132]{harary1977}. Additionally, Theorem~\ref{prop:biparticity-complexity-k2}(2) together with Example~\ref{exem:complexity-kj-k2} recovers the formula for $\mathrm{d}(K_n)$ obtained by Fishburn-Hammer in \cite[Lemma 1, p. 130]{fishburn1996}.  

\begin{proposition}\label{cor:biparticity-formula}
Let $\ell\geq 2$ \begin{enumerate}
       \item[(1)] Let $G$ be a simple graph such that $\mathrm{C}(G;K_{\ell})\geq 2$ and $\chi(G)\geq 2$. We have \[\beta_\ell(G)=\lceil\log_{\ell}\chi(G)\rceil.\]  
       \item[(2)] Let $n\geq 2$. We have \[\mathrm{d}(K_n)=\lceil\log_{2}n\rceil.\]  
        \end{enumerate}
\end{proposition}

We close this section with the following remark, which presents a direct relation between (injective) hom-complexity and the sectional number.

\begin{remark}[(Injective) hom-complexity and sectional number]\label{rem:complexity-sectionalnumber}
Let $G$ and $H$ be graphs. Given a homomorphism $f:G\to H$, we have \[\mathrm{C}(H;G)\leq \mathrm{IC}(H;G)\leq \text{sec}(f).\] Here, $\text{sec}(f)$ denotes the sectional number of $f$ as introduced in \cite{zapata2023}. Specifically, $\text{sec}(f)$ is the least positive integer $k$ such that there exist subgraphs $H_1\ldots,H_k$ of $H$ with $H=H_1\cup\cdots\cup H_k$, and for each $H_i$, there exists a homomorphism $\sigma_i:H_i\to G$ such that $f\circ\sigma_i=\mathrm{incl}_{H_i}$ (and thus each $\sigma_i:H_i\to G$ is an injective homomorphism), where $\mathrm{incl}_{H_i}:H_i\hookrightarrow H$ is the inclusion homomorphism.  
\end{remark}

\section{Designing Optimal Quasi-Homomorphisms}\label{sec:applica}
As mentioned in the Introduction, this work is motivated by a fundamental desire to present a well-defined methodology to address the \aspas{complexity} of the data migration problem \cite{hussein2021}, \cite{spivak2012}. For this reason, in this section, we design optimal quasi-homomorphisms for a concrete example.      

\medskip Consider the graphs $G$ and $H$ depicted in Figure~\ref{figure:exam-G-H}, where we have: \begin{align*}
    V(G)&=\{v_1,\ldots,v_8\}\\ 
    E(G)&=\{v_1v_2,v_1v_3,v_1v_4,v_1v_6,v_2v_3,v_2v_4,v_2v_7,v_3v_4,v_3v_8,v_4v_5\}\\ 
     V(H)&=\{u_1,\ldots,u_6\}\\ 
      E(H)&=\{u_1u_2,u_1u_3,u_1u_4,u_1u_5,u_2u_4,u_2u_5,u_2u_6\}\\
\end{align*} 

\begin{figure}[h!]
\centering
\begin{align*}
 \begin{tikzpicture}
\Vertex[x=1, y=1, size=0.2,label=$v_1$,position=above,color=black]{A} 
\Vertex[x=2, y=2, size=0.2,label=$v_6$,position=above,color=black]{B}
\Vertex[x=1, y=-1, size=0.2,label=$v_2$,position=below,color=black]{C} 
\Vertex[x=2, y=-2, size=0.2,label=$v_7$,position=above,color=black]{D} 
\Vertex[x=-1, y=-1, size=0.2,label=$v_3$,position=below,color=black]{E} 
\Vertex[x=-2, y=-2, size=0.2,label=$v_8$,position=above,color=black]{F} 
\Vertex[x=-1, y=1, size=0.2,label=$v_4$,position=above,color=black]{G} 
\Vertex[x=-2, y=2, size=0.2,label=$v_5$,position=above,color=black]{H} 
\Vertex[x=2, y=-2, size=0.2,label=$G$,position=left,distance=1.5cm,color=black]{I} 
\Edge[color=black](A)(B)
\Edge[color=black](A)(C)
\Edge[color=black](A)(G)
\Edge[color=black](A)(E)
\Edge[color=black](C)(G)
\Edge[color=black](C)(E)
\Edge[color=black](C)(D)
\Edge[color=black](E)(F)
\Edge[color=black](E)(G)
\Edge[color=black](G)(H)
  \end{tikzpicture}  & &
 \begin{tikzpicture}
\Vertex[size=0.2,label=$u_1$,position=above,color=black]{A} 
\Vertex[x=4, size=0.2,label=$u_2$,position=above,color=black]{B}
\Vertex[x=2, y=3, size=0.2,label=$u_3$,position=above,color=black]{C} 
\Vertex[x=2, y=1, size=0.2,label=$u_4$,position=above,color=black]{D} 
\Vertex[x=2, y=-1, size=0.2,label=$u_5$,position=above,color=black]{E} 
\Vertex[x=2, y=-3, size=0.2,label=$u_6$,position=above,color=black]{F} 
\Vertex[x=2, y=-3, size=0.2,label=$H$,position=below,distance=0.5cm,color=black]{G} 
\Edge[color=black](A)(B)
\Edge[color=black](A)(C)
\Edge[color=black](A)(D)
\Edge[color=black](A)(E)
\Edge[color=black](B)(F)
\Edge[color=black](B)(E)
\Edge[color=black](B)(D)
  \end{tikzpicture}  
\end{align*}
\caption{The graphs $G$ and $H$.}\label{figure:exam-G-H}
\end{figure}

We find that $\chi(G)=4$ and $\omega(H)=\chi(H)=3$, leading to the conclusion that $\mathrm{C}(G;H)=2$ (see Corollary~\ref{cor:omega-equal-chi}). 

\medskip Designing an optimal quasi-homomorphism from $G$ to $H$ is greatly simplified by following two main steps:
\begin{itemize}
    \item[S1.] First, we obtain an optimal quasi-homomorphism from $K_4$ to $K_3$.
    \item[S2.] Next, we design an optimal quasi-homomorphism from $G$ to $H$.
\end{itemize}

In Step $S1$, we utilize the proof of Proposition~\ref{prop:upper-bound}. Here, $j=4, i=3, \ell=2, r=1$, so we have \[K_4=\left(K_2\sqcup K_1\sqcup K_1\right)\cup K_{2,1,1},\] where we consider \begin{align*}
    V(K_2)&=\{1,4\},\\
    E(K_2)&=\{14\},\\
    V(K_{2,1,1})&=V_1\sqcup V_2\sqcup V_3\text{ with $V_1=\{1,4\}, V_2=\{2\}$, $V_3=\{3\}$, }\\
    E(K_{2,1,1})&=\{34,31,32,24,21\}.\\
\end{align*} Note that $\mathrm{C}(K_4;K_3)=2$ and the collection $\mathcal{K}=\{f_1:K_2\to K_3,f_2:K_{2,1,1}\to K_3\}$ is an optimal quasi-homomorphism from $K_4$ to $K_3$, where $f_1(1)=1$ and $f_1(4)=2$,  and $f_2(V_i)=\{i\}$ for each $i=1,2,3$. 

\medskip In Step $S2$, we apply the proof of Theorem~\ref{thm:general-upper-bound} along with Theorem~\ref{prop:complexity-subgraphs}(1). Define $f:G\to K_4$ as a $4$-coloring of the graph $G$ defined by $f(v_i)=i$ for each $i=1,2,3,4$, with $f(v_5)=f(v_8)=f(v_7)=1$ and $f(v_6)=2$. Additionally, we recognize that the graph $\left(\{u_1,u_2,u_4\},\{u_1u_2,u_1u_4,u_2u_4\}\right)$ is a copy of $K_3$ in $H$ via the isomorphism $1\leftrightarrow u_1, 2\leftrightarrow u_2, 3\leftrightarrow u_4$. By the proof of Theorem~\ref{prop:complexity-subgraphs}(1), $\mathcal{K}$ induces a quasi-homomorphism from $K_4$ to $H$ given by \[\mathcal{K}'=\{g_1:=\iota\circ f_1:K_2\to H,g_2:=\iota\circ f_2:K_{2,1,1}\to H\},\] where $\iota:K_3\hookrightarrow H$ is the inclusion homomorphism. Subsequently, by the same proof of Theorem~\ref{prop:complexity-subgraphs}(1), $\mathcal{K}'$ induces a quasi-homomorphism from $G$ to $H$ represented by \[\mathcal{M}=\{h_1:G_1\to H,h_2:G_2\to H\},\] where $G_1=f^{-1}(K_2)$ and $G_2=f^{-1}(K_{2,1,1})$; with $h_1=g_1\circ f_\mid$ and $h_2=g_2\circ f_\mid$. Explicitly, we have \begin{align*}
    V(G_1)&=\{v_1,v_4,v_5,v_7,v_8\},\\
    E(G_1)&=\{v_1v_4,v_4v_5\},\\
    V(G_2)&=V(G),\\
    E(G_2)&=\{v_1v_2,v_1v_3,v_1v_6,v_2v_3,v_2v_4,v_2v_7,v_3v_4,v_3v_8\}.\\
\end{align*} Note that $\mathcal{M}$ is optimal because $2=\mathrm{C}(G;H)$. 

\begin{remark}\label{rem:final-remark}
From the construction above, it is evident that \[h_1(G_1)\cup h_2(G_2)=K_3\subsetneq H.\] 
\end{remark}

Finally, we propose the following future work.

\begin{remark}[Future work]\label{rem:future-work}
 \noindent\begin{enumerate}
     \item[(1)] Based on Remark~\ref{rem:final-remark}, we introduce a strong notion of hom-complexity. Given two graphs $G$ and $H$, the \textit{strong hom-complexity} from $G$ to $H$, denoted $\mathrm{sC}(G;H)$, is defined as the least positive integer $k$ such that there exist subgraphs  $G_1,\ldots,G_k$ of $G$ satisfying $G=G_1\cup\cdots\cup G_k$, and for each $G_i$, there exists a homomorphism $f_i:G_i\to H$ such that \[H=f_1(G_1)\cup\cdots\cup f_k(G_k).\] If no such integer $k$ exists, we set $\text{sC}(G;H)=\infty$. It follows that $\mathrm{C}(G;H)\leq \mathrm{sC}(G;H)$. We propose studying this notion of strong hom-complexity further. 
\item[(2)] As mentioned in the introduction, this work is motivated by the data migration problem. One direction for future research is to develop the notion of hom-complexity in higher dimensions$-$for example,  for hypergraphs, simplicial complexes, and related structures. 
\item[(3)] For $n,k\in\mathbb{Z}$, with $n\geq 2$ and $1\leq k\leq n/2$, the Kneser graph $\Gamma_{k,n}$ is the graph whose vertices are all $k$-subsets of $[n]$, with edges connecting pairs of disjoint $k$-subsets. The Kneser conjecture, posed in 1955, became a foundational problem in the field of Combinatorics. It asks to show that $\chi(\Gamma_{k,n})\geq n-2k+2$. In 1978, L\'{a}szl\'{o} Lov\'{a}sz solved the Kneser conjecture by employing the Borsuk–Ulam theorem from Algebraic Topology. Similar and related approaches have also relied on variants of the Borsuk–Ulam theorem. We propose the following question: Is there a graph $H$ such that $\mathrm{C}(\Gamma_{k,n};H)>\lceil\log_{\chi(H)}(n-2k+1)\rceil$ (here $\omega(H)=
\chi(H)$) or $\mathrm{C}(H;\Gamma_{k,n})<\lceil\log_{n-2k+1}\chi(H)\rceil$? In the affirmative case, by Corollary~\ref{cor:complejidad-chro}, we obtain that $\chi(\Gamma_{k,n})\geq n-2k+2$. This would constitute a new solution to the Kneser conjecture.      
 \end{enumerate}   
\end{remark}

\section*{Acknowledgment}
The first author would like to thank grants \#2023/16525-7 and \#2022/16695-7 from the S\~{a}o Paulo Research Foundation (\textsc{fapesp}) for financial support.

\section*{Conflict of Interest Statement}
The authors declare that there are no conflicts of interest.

\bibliographystyle{plain}

\begin{thebibliography}{11}
\bibitem{badson2006} Babson, E., \& Kozlov, D. N.: Complexes of graph homomorphisms. Israel Journal of Mathematics, 152, 285-312 (2006)
\bibitem{wilman}  Cuba Ramos, W. F.: Topología del problema de migración de datos. Ph.D. thesis (in Spanish), FCM, UNMSM, Per\'{u} (2023)
\bibitem{erdos1966} Erdös, P., Goodman, A. W., \& Pósa, L.: The representation of a graph by set intersections. Canadian Journal of Mathematics, 18, 106-112 (1966)
\bibitem{fishburn1996} Fishburn, P. C., \& Hammer, P. L.: Bipartite dimensions and bipartite degrees of graphs. Discrete Mathematics, 160(1-3), 127-148 (1996)
\bibitem{gregory1982} Gregory, D. A., \& Pullman, N. J.: On a clique covering problem of Orlin. Discrete Mathematics, 41(1), 97-99 (1982)
\bibitem{gyarfas1990} Gyárfás, A.: A simple lower bound on edge coverings by cliques. Discret. Math., 85(1), 103-104 (1990)
\bibitem{harary1970} Harary, F.: Covering and packing in graphs, I. Annals of the New York Academy of Sciences, 175(1), 198-205 (1970)
\bibitem{harary1977} Harary, F., Hsu, D., \& Miller, Z.: The biparticity of a graph. Journal of Graph Theory, 1(2), 131-133 (1977)
\bibitem{hell1990} Hell, P., \& Nešetřil, J.: On the complexity of H-coloring. Journal of Combinatorial Theory, Series B, 48(1), 92-110 (1990)
\bibitem{hell2004} Hell, P., \& Nešetřil, J. Graphs and homomorphisms (Vol. 28). Oxford University Press (2004)
\bibitem{hussein2021} Hussein, A. A.: Data migration: Need, strategy, challenges, methodology, categories, risks, uses with cloud computing, and improvements using suggested proposed Model (DMig1). Journal of Information Security, 12, 79-103 (2021)
\bibitem{lovasz1978} Lovász, L.: Kneser's conjecture, chromatic number, and homotopy. Journal of Combinatorial Theory, Series A, 25(3), 319-324 (1978)
\bibitem{orlin1977} Orlin, J.: Contentment in graph theory: covering graphs with cliques. In Indagationes Mathematicae (Proceedings) (Vol. 80, No. 5, pp. 406-424). North-Holland (1977)
\bibitem{schwartz2022} Schwartz, S.: An overview of graph covering and partitioning. Discrete Mathematics, 345(8), 112884 (2022)
\bibitem{spivak2012} Spivak, D. I.: Functorial data migration. Information and Computation, 217, 31-51 (2012) 
\bibitem{west2001} West, D. B.: Introduction to graph theory (Vol. 2). Upper Saddle River: Prentice hall (2001)
\bibitem{zapata2023} Zapata, C. A. I., \& Ramos, W. F. C.: Número seccional de un homomorfismo de grafos. Pesquimat, 26(2), 39-46 (2023)
\bibitem{hammack2011} Hammack, R. H., Imrich, W., \& Klavžar, S.: Handbook of product graphs (Vol. 2). Boca Raton: CRC press (2011)
\end{thebibliography}

\end{document}